\PassOptionsToPackage{dvipsnames,svgnames,x11names}{xcolor}

\documentclass[12pt,reqno]{amsart}

\usepackage{amsmath}
\usepackage{amssymb}
\usepackage{amsthm}
\usepackage{bbm}
\usepackage{mathtools}
\usepackage{nicefrac}

\usepackage[main=english,brazil]{babel}
\usepackage[utf8]{inputenc} 

\usepackage{a4wide}
\usepackage{microtype}

\usepackage[all]{xy}
\usepackage{tikz-cd}

\usepackage{aligned-overset}
\usepackage{enumitem}
\usepackage{multicol}
\usepackage{scalefnt}
\usepackage{xspace}

\usepackage{fixme}
\usepackage{url}

\usepackage{xcolor}
\usepackage{hyperref}
\usepackage{cleveref} 

\hypersetup{
  colorlinks=true,
  linkcolor=BurntOrange,
  citecolor=BurntOrange,
  urlcolor=BurntOrange
}

\theoremstyle{plain}
  \newtheorem*{theorem*}{Theorem}
  \newtheorem*{ps}{\color{RedViolet}The lifting problem}
  \newtheorem*{cs}{\color{RedViolet}The classical lifting problem}
  \newtheorem{theorem}{Theorem}[section]
  \newtheorem{lemma}[theorem]{Lemma}
  \newtheorem{corollary}[theorem]{Corollary}

\theoremstyle{definition}
  \newtheorem{defn*}{Definition}
  \newtheorem{defn}[theorem]{Definition}

\theoremstyle{remark}
  \newtheorem*{remark*}{Remark}
  \newtheorem{remark}[theorem]{Remark}

\newcommand*{\C}{\mathbb C}

\newcommand*{\R}{\mathbb R}

\newcommand*{\alg}{\mathcal}
\newcommand*{\hilb}{\mathfrak}

\newcommand*{\aA}{\alg A}
\newcommand*{\aB}{\alg B}

\newcommand*{\End}{\mathcal L}
\newcommand*{\hH}{\hilb H}

\newcommand{\hA}{\smash{\hat{\aA}}}
\newcommand{\ha}{\smash{\hat{\alpha}}}
\newcommand{\hf}{\smash{\hat{f}}}
\newcommand{\hfi}{\smash{\hat{\varphi}}}
\newcommand{\hG}{\smash{\hat{G}}}
\newcommand{\hg}{\smash{\hat{g}}}
\newcommand{\hiso}{\smash{\hat{A}}}
\newcommand{\hP}{\smash{\hat{P}}}
\newcommand{\hT}{\smash{\hat{\mathbb{T}}}}

\newcommand*{\Cont}{C}

\DeclareMathOperator{\Ad}{Ad}
\DeclareMathOperator{\Aut}{Aut}
\DeclareMathOperator{\Diff}{Diff}
\DeclareMathOperator{\Ext}{Ext}
\DeclareMathOperator{\Fr}{Fr}
\DeclareMathOperator{\Gau}{Gau}

\DeclareMathOperator{\Homeo}{Homeo}
\DeclareMathOperator{\Hom}{Hom}
\DeclareMathOperator{\id}{id}

\DeclareMathOperator{\Irrep}{Irr}

\DeclareMathOperator{\Pic}{Pic}

\DeclareMathOperator{\rep}{Rep}

\DeclareMathOperator{\SO}{SO}

\DeclareMathOperator{\Spin}{Spin}
\DeclareMathOperator{\SU}{SU}
\DeclareMathOperator{\Tr}{Tr}

\newcommand{\aka}{\mbox{a.\,k.\,a.}\xspace}
\newcommand{\acts}{\,.\,}
\newcommand{\cf}{\mbox{cf.}\xspace}
\newcommand{\eg}{\mbox{e.\,g.}\xspace}

\newcommand*{\ie}{\mbox{i.\,e.}\xspace}

\DeclareRobustCommand{\Star}{\texorpdfstring{\ensuremath{^*}\nobreakdash-}{*-}}

\begin{document}

\title[Noncommutative principal bundles and central extensions]{Noncommutative principal bundles \\ and central extensions}




\date{\today}

\author{Stefan Wagner}
\address{Blekinge Tekniska H\"ogskola}
\email{stefan.wagner@bth.se}

\subjclass[2020]{Primary 58B34, 46L85; Secondary 55R10, 16D70}

\keywords{Noncommutative principal bundle, free action, central extension, spin structure}

\begin{abstract}
Motivated by the classical theory of spin structures, we develop a theory for lifting free C\Star dynamical systems, \aka noncommutative principal bundles, along central extensions. 
This theory extends the bundle-theoretic notion of spin structures and yields a complete existence and classification result for such lifts. 
Using factor system techniques and Picard formalism, our approach introduces new invariants and obstruction classes, thereby unifying geometric, cohomological, and operator-algebraic perspectives.
A range of examples demonstrates the scope of the theory.

\end{abstract}

\maketitle

\section{Introduction}
\label{sec:intro}

Riemannian spin geometry lies at the intersection of differential geometry, topology, and global analysis, with principal bundle theory at its conceptual core. 
Its key objects are spin structures, arising from lifts of principal bundles, and the associated Dirac operator, whose analytic and topological features link geometric invariants with index theory and play an essential role in mathematical formulations of quantum field theory.

In the noncommutative setting, spectral triples provide a natural framework for noncommutative Riemannian spin geometry, facilitating the extension of geometric principles by analogy with the classical theory (see, \eg,~\cite{Co94,Co08} and ref.~therein).
However, unlike in the classical setting, the axiomatic formulation of spectral triples does not inherently incorporate noncommutative principal bundles,
which are gaining prominence in a growing range of applications within geometry and mathematical physics (see, \eg, \cite{Cac23,CacMes19,FaLa24,KrBuSt23,SchWa21,SchWa20,To24,Wa23} and ref.~therein).
This presents a compelling opportunity for further investigation.

To set the stage, we briefly review the classical framework of Riemannian spin geometry.
Let $X^n$ be a connected, orientable Riemannian manifold. 
Its frame bundle $\Fr(X)$ is the principal $\SO(n)$-bundle of positively oriented orthonormal frames.
Since $\SO(n)$ admits a central extension $1 \to \mathbb{Z}_2 \to \Spin(n) \to \SO(n) \to 1$ via its universal covering mapping, one may ask whether there exists a principal $\Spin(n)$-bundle $Q$ over $X$ such that $Q/\mathbb{Z}_2 \cong \Fr(X)$---a natural lifting problem for principal bundles.
In the affirmative case, $X$ is called a Riemannian spin manifold, and the principal $\Spin(n)$-bundle $Q$ over $X$ is referred to as a spin structure on $X$.
The obstruction to the existence of a spin structure on $X$ is governed by the second Stiefel--Whitney class $w_2(X) \in H^2(X,\mathbb{Z}_2)$.
Moreover, spin structures are generally not unique: if $w_2(X)$ vanishes, then the inequivalent spin structures are parametrized, up to equivalence, by $H^1(X,\mathbb{Z}_2)$.

Analytically, a spin structure on $X$ yields the spinor bundle $S \to X$ with Hilbert space $L^2(S)$ of square-integrable sections, on which
$C^\infty(X)$ acts by pointwise multiplication,~together with a formally
self-adjoint operator $D$ on $L^2(S)$, the Dirac operator.
If $X$ is compact, $D$ is an unbounded self-adjoint operator with compact
resolvent and bounded commutators with $C^\infty(X)$.
The triple $(C^\infty(X), L^2(S), D)$ then defines a spectral triple, from which the spin manifold can, with suitable additional data, be reconstructed.

Riemannian spin geometry thus rests on the ability to lift principal $\SO(n)$-bundles along the above central extension, placing the lifting problem at the conceptual core of the theory. 
Motivated by this classical perspective, we formulate and analyze the noncommutative analogue of this lifting problem.

\subsection*{Goal of this paper}
\label{sec:goal}

The primary goal of this paper is to develop a general framework for lifting \emph{free C\Star dynamical systems}---natural models for noncommutative principal bundles---along central extensions of the structure group, thereby generalizing the classical notion of a \hyperref[sec:spin_structure]{\emph{spin structure}} to the noncommutative setting.
Given the well-developed theory in the compact case, we restrict attention to \hyperref[sec:free]{free C\Star dynamical systems} $(\aA,G,\alpha)$ with $\aA$ unital and $G$ compact:

\begin{ps}
Let $(\aA,G,\alpha)$ be a free C\Star dynamical system with fixed point algebra $\aB$, and let $1 \rightarrow Z \rightarrow \hG \rightarrow G \rightarrow 1$ be a central extension of compact groups.
Determine whether there exists a free C\Star dynamical system $(\hA,\hG,\ha)$ over $\aB$, containing $\aA$
as a unital C\Star subalgebra, such that ${\hA}^Z = \aA$, considered as free $G$-algebras.
In the affirmative case, the goal is to classify all possible lifts.
\end{ps}

Such a lift, if it exists, will be called a \emph{$\hG$-structure} for $(\aA,G,\alpha)$---or simply a \emph{$\hG$-structure}---in analogy with the classical case of spin structures.

We approach the lifting problem via the \hyperref[sec:facsys]{\emph{factor system theory}} introduced in~\cite{SchWa15,SchWa17,SchWa21}, which offers a versatile framework for the study of free C\Star dynamical systems.
This approach not only offers the structural tools to analyze and construct $\hG$-structures but also yields new invariants that play a central role in their classification.
The main challenge lies in formulating a comprehensive characterization of when a $\hG$-structure exists.

\subsection*{The classical setting}
\label{sec:class_sett}

To contextualize our approach, we recall the lifting problem for classical \hyperref[sec:pb]{principal bundles}, which, however, are not assumed to be smooth or locally trivial, along with its key results.  
While the classical setting provides a rich source of examples and a well-developed obstruction theory, our approach differs markedly in nature.  
Nevertheless, it is guided by the same fundamental question:

\begin{cs}
Let $P$ be a principal $G$-bundle with base space $X$, and let $1 \rightarrow Z \rightarrow \hG \rightarrow G \rightarrow 1$ be a central extension of Lie groups.
Determine whether there exists a principal $\hG$-bundle $\hP$ over $X$ such that $\hP/Z = P$, considered as principal $G$-bundles.
In~the affirmative case, the goal is to classify all possible lifts.
\end{cs}

Throughout the following, all objects are assumed to lie in the smooth category.

\subsubsection*{Existence via crossed modules}
\label{sec:crossed_module}

We begin by recalling a complete cohomological obstruction to the classical lifting problem; for a detailed treatment, see~\cite{LaWa08,Neeb06}.
Let $P$ be a principal $G$-bundle with base space $X$.
Consider the associated global Atiyah sequence:
\begin{equation}
\label{eq:atiyah}
	1 \longrightarrow \Gau(P) \longrightarrow \Aut(P) \longrightarrow \Diff(X)_{[P]}\longrightarrow 1, 
\end{equation}
where $\Diff(X)_{[P]} \subseteq \Diff(X)$ denotes the open Lie subgroup consisting of diffeomorphisms that preserve the bundle class of $[P]$ under pullbacks.
Additionally, let $1 \to Z \to \hG \to G \to 1$ be a central extension of Lie~groups.
This data gives rise to a new central extension of (infinite-dimensional) Lie groups:
\begin{equation}
\label{eq:ind}
	1 \longrightarrow C^{\infty}(X,Z) \longrightarrow \widehat{\Gau}(P)\longrightarrow \Gau(P) \longrightarrow 1, 
\end{equation}
where 
\begin{equation*}
	\widehat{\Gau}(P) := \{f \in C^{\infty}(P,\hG) : (\forall p\in P)(\forall g\in G) \, f(p \acts g) = gf(p)g^{-1}\}.
\end{equation*}
Equations~\eqref{eq:atiyah} and~\eqref{eq:ind} define a \emph{crossed module} $\mu: \smash{\widehat{\Gau}}(P) \to \Aut(P)$. 
Its characteristic class lies in $H^3(\Diff(X)_{[P]},C^{\infty}(X,Z))$.
Applying the canonical derivation map yields a 3-de Rham cohomology class on $X$ with values in the Lie algebra of $Z$, which serves as an obstruction to the existence of $\hP$ (see~\cite[Prop.~VI.3]{Neeb06}).
Laurent-Gengoux and~Wagemann~\cite{LaWa08} further showed that this class constitutes a complete obstruction to the classical lifting problem.
In particular, if such a $\hP$ exists, then $\Gau(\hP)$ is isomorphic to $\smash{\widehat{\Gau}}(P)$.

Beyond these obstruction results, to the best of our knowledge, no systematic classification results are available in this general setting, which highlights the novelty of our approach.

\subsubsection*{Spin structures}
\label{sec:spin_structure}

We now establish the link between the classical lifting problem and the notion of a spin structure.

\begin{defn}
\label{def:spin}
Let $(P,X,\SO(n),r_P,q_P)$, with $n \geq 3$, be a principal bundle.
Moreover, let $\Spin(n)$ be the universal cover of $\SO(n)$, and let $\phi: \Spin(n) \to \SO(n)$ be the corresponding 2-fold covering map.
A \emph{spin structure} on $P$ is a pair $((Q,X,\Spin(n),r_Q,q_Q),\Phi)$ consisting~of:
\begin{itemize}
\item
	a principal bundle $(Q,X,\Spin(n),r_Q,q_Q)$,  and
\item
	a 2-fold covering map $\Phi: Q \to P$ for which the following diagram commutes:
	\[
	\begin{tikzcd}
		Q \times \Spin(n) \arrow{r}{r_Q} \arrow[swap]{dd}{\Phi \times \phi} & Q \arrow{rd}{q_Q} \arrow{dd}{\Phi} \\
		& & X
		\\
		P \times \SO(n) \arrow{r}{r_P} & \arrow{ru}{q_P} P
	\end{tikzcd}
	\]
\end{itemize}
\end{defn}

The existence of a spin structure is thus a special case of the classical lifting problem for principal $\SO(n)$-bundles and the central extension $1 \rightarrow \mathbb{Z}_2 \rightarrow \Spin(n) \rightarrow \SO(n) \rightarrow 1$.
Its solution is encapsulated in the following result:

\begin{theorem}[{see, \eg,~\cite{Ply86} and~\cite[p.~40]{friedrich2000}}]
\label{thm:topspin}
A principal $\SO(n)$-bundle $P$ with base space~$X$ admits a spin structure if and only if its \emph{second Stiefel--Whitney class} $w_2(P) \in H^2(X,\mathbb{Z}_2)$ vanishes.
In that case, the inequivalent spin structures are parametrized, up to equivalence, by $H^1(X,\mathbb{Z}_2)$.
\end{theorem}

\begin{remark}
An orientable Riemannian manifold $X$ is said to be a \emph{Riemannian spin manifold} if its frame bundle admits a spin structure.
\end{remark}

\subsection*{Organization of the paper}

In Section~\ref{sec:prelim}, we set out the necessary definitions, notations, and results that will be referenced throughout the paper.
In particular, Section~\ref{sec:facsys} reviews the factor system theory underlying our approach.

In Section~\ref{sec:extension_problem}, which forms the heart of the paper, we resolve the~\hyperref[sec:goal]{lifting problem} formulated at the outset. 
More precisely, our strategy unfolds as follows:

In Section~\ref{sec:analyzing_G_structures}, we initiate the study of a given~\hyperref[sec:goal]{$\hG$-structure} by extracting and analyzing its underlying structural data.
In particular, we identify key invariants and~characterize its ($Z$-)gauge group via crossed homomorphisms (see~Theorem~\ref{thm:gau}).

In Section~\ref{sec:classification}, we establish a classification of such $\hG$-structures.
For each fixed equivariant Picard homomorphism admitting a realization, we show that the corresponding equivalence classes of $\hG$-structures form a principal homogeneous
space under $H^2_\Delta(Z^*,\mathcal{U}(Z(\aA)^G))$ (see Corollary~\ref{cor:class1}).

In Section~\ref{sec:constructing_G_structures}, we resolve the existence problem by presenting a systematic, step-by-step construction of a $\hG$-structure, which yields the main result of the paper (see Theorem~\ref{thm:main}). 
This proceeds in four steps: we first construct a candidate algebra $\hA$ extending $\aA$; then lift $\alpha$ to a $\hG$-action on $\hA$ (see Lemma~\ref{lem:lift} and Theorem~\ref{thm:grp_hom_class}); ensure the continuity of this action (see Lemma~\ref{lem:continuity}); and finally prove that the resulting C\Star dynamical system is free (see Theorem~\ref{thm:freeness}).

In Section~\ref{sec:examples}, we present a range of examples illustrating the results of Section~\ref{sec:extension_problem}. 
More precisely:  
In Section~\ref{sec:toyspin}, we give a simple example illustrating the underlying concepts;
in Section~\ref{sec:qtori}, we study coverings of quantum tori, which yield natural instances of the lifting problem; 
in Section~\ref{sec:connes-landi}, we develop a concrete noncommutative $\Spin(3)$-structure based on the Connes--Landi spheres, and in Section~\ref{sec:quantumspin}, we present the general framework for noncommutative spin structures.
Within this framework, we address the existence and classification of spin structures for a given \emph{noncommutative frame bundle};
in Section~\ref{sec:topspin}, we apply our results to classical \hyperref[sec:pb]{principal bundles}, thereby analyzing the~\hyperref[sec:class_sett]{classical lifting problem};  
and finally, in Section~\ref{sec:QT3}, we construct a noncommutative $\mathbb{T}^3$-structure from a suitable quantum 3-torus.

\subsection*{Future directions}

In future work, we aim to construct new examples of noncommutative Riemannian spin geometries, extending the classical theory by analogy:
Assume the setting of the~\hyperref[sec:goal]{lifting problem}, and let $(\hA,\hG,\ha)$ be a $\hG$-structure.  
Let $(\hat{\sigma},\hat{V})$ be a finite-dimensional unitary representation of~$\hG$, and consider the associated C\Star correspondence
\begin{equation*}
    \Gamma_{\hA}(\hat{\sigma}) := \{x \in \hA \otimes \hat{V} \mid (\ha_{\hg} \otimes \hat{\sigma}_{\hg})(x) = x \text{ for all } \hg \in \hG\},
\end{equation*}
which serves as a noncommutative analogue of the spinor module.
The goal is to construct a spectral triple on the C\Star algebra of adjointable operators on $\Gamma_{\hA}(\hat{\sigma})$, using the geometry of the $\hG$-structure $(\hA,\hG,\ha)$ and a given spectral triple on the fixed point algebra~$\aB$.  
Such constructions may naturally arise from noncommutative frame bundles, as described in~\cite{Wa23}.

We also aim to adapt the notion of a crossed module, which---as previously noted---plays a central role in the~\hyperref[sec:class_sett]{classical lifting problem} (see also~\cite{MickWa16,Neeb07}).
Given the close connections between crossed modules, gerbes, and higher gauge theory, this direction provides a natural pathway toward developing a theory of quantum gerbes and advancing noncommutative gauge theory.
At present, however, such an approach appears to~be~technically out of reach.

Beyond these directions, we expect that with further technical refinements, the present results will extend to non-Abelian extensions of compact groups.

\section{Preliminaries and notation}
\label{sec:prelim}

Our study deals with noncommutative principal bundles.
This section exhibits the most fundamental definitions and notations in use.

At first we provide some standard references.
For a comprehensive introduction to the theory of fibre bundles, especially principal bundles and (their associated) vector bundles, we refer to Husem\"oller's monograph~\cite{Huse94} and the influential exposition~\cite{Kob69I} by Nomizu and Kobayashi. 
For a recent account of the theory of Non-commutative Riemannian spin geometry, we recommend the excellent volume~\cite{gracia2000} by Gracia-Bond{\'\i}a, Varilly, and Figueroa.
The theory of operator algebras is well covered in the authoritative works of Blackadar~\cite{BB06} and Pedersen~\cite{Ped18}.
Finally, for background on abstract group cohomology---specifically cochains, cocycles, and coboundaries---we refer to the classical text~\cite{MacLane95} by Mac Lane.

We now fix some general notation.
The identity element of a group is denoted by $e$.
For a unital C\Star algebra~$\aA$, we write $1_\aA$ for its unit, $\mathcal{U}(\aA)$ for its group of unitary elements, and $Z(\aA)$ for its center.
Unless stated otherwise, $\otimes$ denotes the minimal C\Star tensor product.

\subsection{Principal bundles}
\label{sec:pb}

A \emph{principal bundle} is a quintuple $(P,X,G,r,q)$, where 
\begin{itemize}
\item	
	$P$ and $X$ are locally compact spaces,
\item
	$G$ is a locally compact group, the \emph{structure group},
\item
	$r: P \times G \to P$ is a continuous right action of $G$ on $P$, which is free and proper,
\item
	$q: P \to X$ is a continuous and open surjection,
\end{itemize}
subject to the following conditions:
\begin{itemize}
\item
	$q(r(p,g)) = q(p)$ for all $p \in P$ and $g \in G$;
\item
	for each $x \in X$, the fibre $q^{-1}(x)$ is $G$-equivariantly homeomorphic to $G$.
\end{itemize}
When no confusion can arise, we suppress the maps $r$ and $q$ and simply say that $P$ is a \emph{principal $G$-bundle over $X$}. 
In this case we write $p \acts g := r(p,g)$ for $p \in P$ and $g \in G$.

In the smooth category, $P$ and $X$ are smooth manifolds, $G$ is a Lie group, $r$ is smooth, and $q$ is a smooth
surjective submersion.
Local triviality follows from the slice theorem (see,~\eg,~\cite[Secs.~2.2 and~4.1]{Pal61}).

\subsection{Representations of groups}

Let $G$ be a compact group.  
All representations of $G$ are assumed to be finite-dimensional and unitary, unless stated otherwise. 
A representation $\sigma: G \to \mathcal{U}(V_\sigma)$ is denoted by $(\sigma,V_\sigma)$, or simply by $\sigma$.
In particular, we use $1$ to denote the trivial representation.

The set of equivalence classes of irreducible representations of $G$ is denoted by $\Irrep(G)$.
By slight abuse of notation, elements of $\Irrep(G)$ are also referred to by $\sigma$, and when necessary, a representative representation $(\sigma, V_\sigma)$ is chosen for $\sigma \in \Irrep(G)$.

When $G$ is compact Abelian, its Pontryagin dual is denoted by $G^*$, with elements typically written as $\chi$.
The notation $\hG$ is reserved throughout for the group in a central extension of~$G$.


\subsection{C\Star dynamical systems}
\label{sec:cstar}

Let $\aA$ be a unital C\Star algebra, let $G$ be a compact group, and let $\alpha: G \to \Aut(\aA)$ be a strongly continuous group homomorphism.  
We refer to such a triple $(\aA,G,\alpha)$ as a \emph{C\Star dynamical system} and adopt the standard shorthand $\alpha_g := \alpha(g)$ for $g \in G$.
The corresponding fixed point algebra is 
\begin{equation*}
    \aA^G := \{x \in \aA : (\forall g \in G) \, \alpha_g(x) = x\},
\end{equation*}
which we typically denote by $\aB$.
When emphasizing its role, we may refer to $(\aA,G,\alpha)$ as a \emph{C\Star dynamical system over $\aB$}.
Conversely, when the focus is on $\aA$, we may simply say that $\aA$ is a \emph{$G$-algebra}, leaving the action $\alpha$ implicit.

For finite-dimensional Hilbert spaces $V$ and $H$, the action $\alpha$ is extended to $\aA \otimes \End(V,H)$ by $\alpha_g \otimes \id_{\End(V,H)}$, $g \in G$, and is again denoted by $\alpha_g$.
Thus, $\alpha$ acts only on the $\aA$-factor.
The same convention is used for other actions appearing below.

The conditional expectation $P_1$ onto $\aB$ defines a right $\aB$-valued inner product on $\aA$ by
\begin{equation}
\label{eq:innprod}
	\langle x, y \rangle_\aB := P_1(x^*y) = \int_G \alpha_g(x^*y) \; dg,
	\qquad
	x,y \in \aA.
\end{equation}
Completing $\aA$ with respect to the resulting norm yields a right Hilbert $\aB$-module $L^2(\aA)$, on which $\aA$ acts faithfully by adjointable operators via the \Star representation $\lambda: \aA \to \End(L^2(\aA))$, densely defined by $\lambda(x)y := x \cdot y$ for $y \in \aA \subseteq L^2(\aA)$.
Thus we may view $\aA$ as the subalgebra $\lambda(\aA) \subseteq \End(L^2(\aA))$.

For each $g \in G$, let $U_g$ be the unitary operator on $L^2(\aA)$, densely defined by $U_g(x) := \alpha_g(x)$ for $x \in \aA \subseteq L^2(\aA)$.
The resulting map $G \ni g \mapsto U_g \in \mathcal{U}(L^2(\aA))$ is strongly continuous and~implements the \Star automorphisms $\alpha_g$, $g \in G$, via
\begin{equation*}
    \lambda(\alpha_g(x)) = U_g \lambda(x) U_g^*,
    \qquad
	x \in \aA.
\end{equation*}

Like any strongly continuous representation of~$G$, the algebra $\aA$ admits a decomposition into isotypic components $A(\sigma) := P_\sigma(\aA)$, $\sigma \in \Irrep(G)$, where 
\begin{equation*}
	P_\sigma(x) := \dim(\sigma) \cdot \int_G \Tr(\sigma_g^*) \, \alpha_g(x) \; dg,
    \qquad
	x \in \aA.
\end{equation*}
This decomposition is realized by the algebraic direct sum
\begin{equation}
\label{eq:A_f}
	\aA_f:= \bigoplus_{\sigma \in \Irrep(G)} A(\sigma),
\end{equation}
which forms a dense \Star subalgebra of $\aA$. 

If $G$ is compact Abelian, then for each $\chi \in G^*$, the isotypic component 
\begin{equation*}
    A(\chi) = \{x \in \aA : (\forall g \in G) \, \alpha_g(x) = \chi(g) \cdot x\}
\end{equation*}
is \emph{almost} a Morita equivalence $\aB$-bimodule, equipped with the canonical $\aB$-bimodule~structure and the inner products 
\begin{equation*}
    {}_\aB\langle x, y \rangle := x y^*
    \qquad
    \text{and}
    \qquad
    \langle x, y \rangle_\aB := x^* y,
    \qquad
    x,y \in A(\chi).
\end{equation*}
However, the linear spans of these inner products need not be dense in~$\aB$.
Despite this, for each $\chi \in G^*$, the norms induced by the left and right inner products on $A(\chi)$ coincide, and the resulting topology agrees with the subspace topology inherited from $\aA$.

\subsection{Freeness}
\label{sec:free}

A C\Star dynamical system $(\aA,G,\alpha)$ is called \emph{free} if the \emph{Ellwood map} 
\begin{equation*}
	\Phi: \aA \otimes_{\text{alg}} \aA \rightarrow \Cont(G,\aA),
	\qquad
	\Phi(x \otimes y)(g) := x \alpha_g(y)
\end{equation*}
has dense range with respect to the canonical C\Star norm on $\Cont(G,\aA)$.
This notion, introduced by Ellwood~\cite{Ell00} in the setting of quantum group actions, is equivalent to Rieffel's notion of saturatedness~\cite{Rieffel91} and to the Peter--Weyl--Galois condition~\cite{BaCoHa15}.

According to \cite[Prop.~7.1.12 \& Thm.~7.2.6]{Phi87}, a continuous action $r:P\times G\rightarrow P$ of a compact group $G$ on a compact space $P$ is free in the usual sense if and only if the induced C\Star dynamical system $(\Cont(P),G,\alpha)$, where 
\begin{equation*}
    \alpha_g(f)(p) := f(r(p,g))
\end{equation*}
for all $p \in P$ and $g \in G$, is free in the sense of Ellwood. 
Free C\Star dynamical systems thus provide a natural framework for noncommutative principal bundles.

Among free C\Star dynamical systems $(\aA,G,\alpha)$, the particularly tractable \emph{cleft} systems are characterized by the existence, for each $\sigma \in \Irrep(G)$, of a unitary
$u(\sigma) \in \aA \otimes \End(V_\sigma)$ satisfying $\alpha_g(u(\sigma)) = u(\sigma) \sigma_g$ for all $g \in G$ (see~\cite[Def.~4.1]{SchWa16}).
In particular, all associated noncommutative vector bundles are trivial.
Despite this constraint, they support a surprisingly rich array of noncommutative phenomena.
Cleft C\Star dynamical systems will occur in several examples in Section~\ref{sec:examples}.

A characterization of freeness particularly useful below is provided by \cite[Lem.~3.3]{SchWa17}.
Specialized to compact group actions and expressed in terms of isometries by taking adjoints of the coisometries used there, it states that a C\Star dynamical system $(\aA,G,\alpha)$ is free if and only if, for each $\sigma \in \Irrep(G)$, there exist a finite-dimensional Hilbert space $\hH_\sigma$ and an isometry $s(\sigma) \in \aA \otimes \End(V_\sigma,\hH_\sigma)$
satisfying $\alpha_g(s(\sigma)) = s(\sigma) \sigma_g$ for all $g \in G$.
Here and above, $\sigma_g$ is identified with $1_\aA \otimes \sigma_g$.

If $G$ is compact Abelian, the above characterization states that freeness is equivalent to the existence, for each $\chi \in G^*$, of a finite-dimensional Hilbert space $\hH_\chi$ and an isometry
$s(\chi) \in \aA \otimes \End(\C,\hH_\chi)$ satisfying $\alpha_g(s(\chi)) = \chi(g) \cdot s(\chi)$ for all $g \in G$.
For each $\chi \in G^*$, any such isometry gives the following description of the corresponding isotypic component:
\begin{equation*}
    A(\chi) = \{y s(\chi) : y \in \aB \otimes \End(\hH_\chi,\C)\}.
\end{equation*}
Indeed, for $x \in A(\chi)$, the product $x s(\chi)^*$ is fixed by the action, and $x  =(x s(\chi)^*) s(\chi)$.
Conversely, each product $y s(\chi)$ with $y \in \aB \otimes \End(\hH_\chi,\C)$ belongs to $A(\chi)$.
This description will be used below in the analysis of $\hG$-structures.

\subsection{Factor systems}
\label{sec:facsys}

Let $(\aA,G,\alpha)$ be a free C\Star dynamical system with fixed point algebra $\aB$.
For each~$\sigma \in \Irrep(G)$,~let $\hH_\sigma$ be a finite-dimensional Hilbert space, and let $s(\sigma) \in \aA \otimes \End(V_\sigma,\hH_\sigma)$ be an isometry such that $\alpha_g\bigl(s(\sigma)\bigr)=s(\sigma) \cdot \sigma_g$ for all $g \in G$.
For $1 \in \Irrep(G)$, we set $\hH_1 := \C$ and $s(1) := 1_\aA$.
A distinguishing feature of $(\aA,G,\alpha)$ is the \emph{factor system} associated with the isometries $s(\sigma)$, $\sigma \in \Irrep(G)$ (see~\cite[Def.~4.1]{SchWa17}), whose construction and basic properties are recalled below.

Let $\rep(G)$ denote the C\Star tensor category of representations of $G$.

First, we naturally extend the assignments $\sigma \mapsto \hH_\sigma$ and $\sigma \mapsto s(\sigma)$ to arbitrary representations $\sigma \in \rep(G)$ by taking direct sums with respect to irreducible subrepresentations.
Second, for each $\sigma \in \rep(G)$, we define the \Star homomorphism 
\begin{equation*}
	\gamma_\sigma: \aB \to \aB \otimes \End(\hH_\sigma),
	\qquad
	\gamma_\sigma(b) := s(\sigma) (b \otimes 1_{V_\sigma}) s(\sigma)^*.
\end{equation*}
Third, for each pair $(\sigma,\pi) \in \rep(G) \times \rep(G)$, we introduce the element
\begin{equation*}
	\omega(\sigma,\pi) := s(\sigma) s(\pi) s(\sigma \otimes \pi)^* \in \aB \otimes \End(\hH_{\sigma \otimes \pi},\hH_\sigma \otimes \hH_\pi).
\end{equation*}
Here and below, operator products are understood with the appropriate identity amplifications on the Hilbert space factors.
Likewise, for each $\sigma \in \rep(G)$, the map $\gamma_\sigma$ is amplified by the identity on any additional operator factors.

The following \emph{twisted cocycle conditions} hold:
\begin{align}
	\label{eq:ranges_sys}
	\omega(\sigma,\pi)^* \omega(\sigma,\pi) = \gamma_{\sigma \otimes \pi}(1_\aB), 
	& \qquad 
	\omega(\sigma,\pi) \omega(\sigma,\pi)^* = \gamma_\sigma \bigl(\gamma_\pi(1_\aB) \bigr), 
	\\
	\label{eq:coaction_sys}
	\omega(\sigma, \pi) \gamma_{\sigma \otimes \pi}(b) &= \gamma_\sigma\bigl( \gamma_\pi(b) \bigr) \omega(\sigma, \pi),
	\\ 
	\label{eq:cocycle_sys}
	\omega(\sigma, \pi) \omega(\sigma \otimes \pi, \rho) 
		&= \gamma_\sigma \bigl( \omega(\pi, \rho) \bigr) \omega(\sigma, \pi \otimes \rho)
\end{align}
for all $\sigma, \pi, \rho \in \rep(G)$ and $b \in \aB$ (see~\cite[Lem.~4.3]{SchWa17}). 
The triple $(\hH,\gamma,\omega)$, representing the above families, is called the \emph{factor system of $(\aA, G, \alpha)$ associated with the isometries $s(\sigma)$, $\sigma \in \Irrep(G)$}.
When an explicit reference to the isometries is unnecessary, it is simply referred to as a \emph{factor system of $(\aA, G, \alpha)$}.

Importantly, the notion of a factor system of $(\aA,G,\alpha)$ depends solely on data associated with the structure group $G$ and the fixed point algebra $\aB$, motivating the following definition:

\begin{defn}
\label{def:absfacsys}
Let $G$ be a compact group and let $\aB$ be a unital C\Star algebra.
A \emph{factor~system for $(G,\aB)$} is defined as a triple $(\hH,\gamma,\omega)$ consisting of:
\begin{itemize}
\item
	a family of finite-dimensional Hilbert spaces $\hH_\sigma$, $\sigma \in \Irrep(G)$,
\item
	a family of \Star homomorphisms $\gamma_\sigma: \aB \to \aB \otimes \End(\hH_\sigma)$, $\sigma \in \Irrep(G)$, and 
\item
	 a family of elements $\omega(\sigma, \pi)$ in $\aB \otimes \End(\hH_{\sigma \otimes \pi}, \hH_\sigma \otimes \hH_\pi)$, $\sigma, \pi\in \Irrep(G)$,
\end{itemize}
satisfying Equations~\eqref{eq:ranges_sys},~\eqref{eq:coaction_sys},~\eqref{eq:cocycle_sys},  along with the normalization conditions $\hH_1 = \C$, $\gamma_1 = \id_\aB$, and $\omega(1, \sigma) = \gamma_\sigma(1_\aB) = \omega(\sigma, 1)$ for all $\sigma \in \Irrep(G)$.
Note that we have implicitly employed functorial extensions of the families $\hH_\sigma$, $\gamma_\sigma$, and $\omega(\sigma, \pi)$, $\sigma, \pi \in \Irrep(G)$.
\end{defn}

Clearly, each factor system of $(\aA,G,\alpha)$ is also a factor system for $(G,\aB)$.
Furthermore, for any unital C\Star algebra $\aB$ and any compact group $G$, each factor system for $(G,\aB)$ gives rise to a free C\Star dynamical system with structure group $G$ and fixed point algebra $\aB$, in which it appears as one of its factor systems (see~\cite[Sec.~5]{SchWa17}).

The following allows for a comparison of factor systems:

\begin{defn}
\label{def:conjugacy}
Let $\aB$ be a unital C\Star algebra and let $G$ be a compact group.
Two factor systems $(\hH,\gamma,\omega)$ and $(\hH',\gamma',\omega')$ for $(G,\aB)$ are called \emph{conjugated} if there exist partial isometries $v(\sigma) \in \aB \otimes \End(\hH_\sigma,\hH'_\sigma)$, $\sigma \in \Irrep(G)$, normalized by $v(1) = 1_\aB$, such that
\begin{gather*}
	\Ad[v(\sigma)] \circ \gamma_\sigma = \gamma_\sigma',
	\qquad 
	\Ad[v(\sigma)^*] \circ \gamma_\sigma' = \gamma_\sigma, 
	\\
	v(\sigma) \gamma_\sigma\bigl( v(\pi) \bigr)  \omega(\sigma, \pi)
	= \omega'(\sigma, \pi) v(\sigma \otimes \pi)
\end{gather*}
for all $\sigma, \pi \in \Irrep(G)$.
In this case, we write $(\hH',\gamma', \omega') = v(\hH, \gamma, \omega)v^*$ or $(\hH,\gamma, \omega) \sim (\hH', \gamma', \omega')$ when no reference to the partial isometries is needed.
Note that we have implicitly utilized a functorial extension of the family  $v(\sigma)$, $\sigma \in \Irrep(G)$.
\end{defn}

All factor systems of $(\aA,G,\alpha)$ are conjugated (see~\cite[Lem. 4.3]{SchWa17}).
Furthermore, for any compact group $G$ and any unital C\Star algebra $\aB$, the conjugacy classes of factor systems for $(G,\aB)$ correspond bijectively to the equivalence classes of free C\Star dynamical systems with structure group $G$ and fixed point algebra $\aB$ (see~\cite[Thm.~5.6]{SchWa17}).

\subsubsection*{Invariants via factor systems}

Given a factor system $(\hH,\gamma,\omega)$ of $(\aA,G,\alpha)$, Equations~\eqref{eq:ranges_sys} and~\eqref{eq:coaction_sys} naturally suggest considering the $K$-theory of $\aB$ and, in particular, the induced positive group homomorphisms $K_0(\gamma_\sigma): K_0(\aB) \to K_0(\aB)$, for $\sigma \in \rep(G)$.
These maps depend only on the conjugacy class of the factor system and thus give rise to invariants of $(\aA,G,\alpha)$. 
Moreover, the assignment $\sigma \mapsto K_0(\gamma_\sigma)$ defines a functorial structure:
\begin{itemize}
\item 
    $K_0(\gamma_{\sigma \oplus \pi}) = K_0(\gamma_\sigma) + K_0(\gamma_\pi)$, 
\item
    $K_0(\gamma_{\sigma \otimes \pi}) = K_0(\gamma_\sigma) \circ K_0(\gamma_\pi)$
\end{itemize}
for all  $\sigma, \pi \in \rep(G)$.

\subsection{The equivariant Picard group}
\label{sec:pic}

Let $(\aA,G,\alpha)$ be a C\Star dynamical system.

A \emph{Morita self-equivalence over $(\aA,G,\alpha)$} is a Morita equivalence $\aA$-bimodule (also called an imprimitivity bimodule) $M$ equipped with a strongly continuous $G$-action $\mu$ on $M$ by linear automorphisms such that
\begin{equation}
\label{eq:MSE}
    \alpha_g(\langle x,y \rangle) = \langle \mu_g(x), \mu_g(y) \rangle
\end{equation}
for all $g \in G$ and $x,y \in M$; here, the equation is understood to hold for both inner products.

Two Morita self-equivalences $M$ and $M'$ over $(\aA,G,\alpha)$ are said to be \emph{equivalent} if there exists a $G$-equivariant Morita equivalence $\aA$-bimodule isomorphism $\psi:M \to M'$, \ie, an $\aA$-bimodule isomorphism which is $G$-equivariant and preserves the inner products.
The equivalence class of a Morita self-equivalence $M$ over $(\aA,G,\alpha)$ is denoted by $[M]$.

The set of all equivalence classes of Morita self-equivalences over $(\aA,G,\alpha)$ forms a group under the internal tensor product $\otimes_\aA$.
This group, denoted by $\Pic_G(\aA)$, is called the \emph{equivariant Picard group} of $(\aA,G,\alpha)$.
If $G$ is trivial, then this group coincides with the usual Picard group of $\aA$ (see, \eg,~\cite[Sec.~3]{BrGrRi77}).

Forgetting the $G$-action defines a group homomorphism $F: \Pic_G(\aA) \to \Pic(\aA)$, sending the equivariant class $[M]$ to the class of its underlying Morita equivalence $\aA$-bimodule.
This homomorphism need not be injective.




For a Morita self-equivalence $M$ over $(\aA,G,\alpha)$, the group of $G$-equivariant Morita equivalence $\aA$-bimodule automorphisms of $M$ is denoted by $\Aut_G(M)$. 
The group of Morita equivalence $\aA$-bimodule automorphisms of $M$, without an equivariance
requirement, is denoted by $\Aut(M)$.

The following statements, although well known to specialists, appear to be only implicitly recorded in the literature:

\begin{lemma}[{\cf~\cite[Prop.~5.4 \& Cor.~5.5]{SchWa15}}]
\label{lem:MSE}
Let $(\aA,G,\alpha)$ be a C\Star dynamical system, and let $M$ be a Morita self-equivalence over $(\aA,G,\alpha)$.
Then the following assertions hold:
\begin{enumerate}[label={\arabic*.},ref=\ref{lem:MSE}.{\arabic*}]
\item\label{lem:MSE:autME}
    The map $\phi_M: \mathcal{U}(Z(\aA)^G) \to \Aut_G(M)$ defined by $\phi_M(u)(m) := um$ is a group~isomorphism.
\item\label{lem:MSE:isoME}
    There is a unique group automorphism $\Delta_M$ of $\mathcal{U}(Z(\aA)^G)$ determined by the relation $\Delta_M(u)m = mu$ for all $m \in M$.
\end{enumerate}
\end{lemma}
\begin{proof}
Let $\psi \in \Aut_G(M)$.
By~\cite[Prop.~5.4]{SchWa15}, there exists $u \in \mathcal{U}(Z(\aA))$ such that $\psi(m) = um$ for all $m \in M$.
Here, we only show that $\psi$ must in fact come from an element $u \in \mathcal{U}(Z(\aA)^G)$.  
The remainder of the proof then proceeds exactly as in~\cite[Prop.~5.4 \& Cor.~5.5]{SchWa15}.

Let $\mu$ be a strongly continuous $G$-action on $M$ by linear automorphisms satisfying~Equation~\eqref{eq:MSE}. 
Then, for $g \in G$ and $m,m' \in M$,
\begin{equation*}
    {}_\aA\langle \mu_g(\psi(m)), \mu_g(m') \rangle
    =
    \alpha_g(u) \alpha_g({}_\aA\langle m, m' \rangle).
\end{equation*}
On the other hand, using the $G$-equivariance of $\psi$,
\begin{equation*}
    {}_\aA\langle \mu_g(\psi(m)), \mu_g(m') \rangle
    =
    {}_\aA\langle \psi(\mu_g(m)), \mu_g(m') \rangle
    =
    u \alpha_g({}_\aA\langle m, m' \rangle).
\end{equation*}
Since $M$ is a Morita equivalence $\aA$-bimodule, this implies that $\alpha_g(u) = u$ for all $g \in G$, \ie, $u \in \mathcal{U}(Z(\aA)^G)$, as claimed.
\end{proof}

\section{The lifting problem}
\label{sec:extension_problem}

In this section, we outline our approach to the \hyperref[sec:goal]{lifting problem}.
We begin by fixing the following data:
\begin{itemize}
\item
    a free C\Star dynamical system $(\aA,G,\alpha)$ with fixed point algebra~$\aB$;
\item
	a central extension $1 \rightarrow Z \xrightarrow{\iota} \hG \xrightarrow{q} G \rightarrow 1$ of compact groups.
\end{itemize}

The central question is whether there exists a free C\Star dynamical system $(\hA,\hG,\ha)$ over~$\aB$, containing $\aA$ as a unital C\Star subalgebra, such that $\hA^Z = \aA$ as $G$-algebras.
Here, the induced action of $G \cong \hG/Z$ on $\hA^Z$ is given by $\bar{\alpha}_g(x) := \ha_{\hg}(x)$, where $\hg$ is any lift of $g$ under $q$.
Normality of $Z$ ensures that $\hA^Z$ is preserved, and $Z$-invariance makes this definition independent of the chosen lift.
The induced action is strongly continuous because $q: \hG \to G$ is a quotient map.
The equality as $G$-algebras means that $\bar{\alpha}_g = \alpha_g$ for all $g \in G$.

Equivalently, in the terminology introduced in the introduction, we ask whether $(\aA,G,\alpha)$ admits a \hyperref[sec:goal]{$\hG$-structure}.
When such lifts exist, our goal is to classify them.

\subsection{\texorpdfstring{Analyzing a given $\mathbf{\hG}$-structure}
{Analyzing a given G-hat-structure}}
\label{sec:analyzing_G_structures}

Suppose we are given a $\hG$-structure $(\hA,\hG,\ha)$.
Our first objective is to extract the essential structural data it encodes.

By the defining compatibility with the induced $G$-action on $\hA^Z = \aA$, the action $\ha$ satisfies $\ha_{\hg}\rvert_\aA=\alpha_{q(\hg)}$ for all $\hg \in \hG$.

Restricting $\ha$ to $Z$, and denoting it by $\theta$, yields a C\Star dynamical system $(\hA,Z,\theta)$ with fixed point algebra $\aA$. 
Let $\hiso(\chi)$, $\chi \in Z^*$, denote the corresponding isotypic components.
The centrality of $Z \subseteq \hG$ ensures that $\theta$ and $\ha$ commute, \ie, $\theta_z \circ \ha_{\hg} = \ha_{\hg} \circ \theta_z$ for all $z \in Z$ and $\hg \in \hG$, so that each $\hiso(\chi)$, for $\chi \in Z^*$, is invariant under $\ha$.

Crucially, $(\hA,Z,\theta)$ remains free~\cite[Prop.~3.17]{SchWa15}.
Since $Z$ is compact Abelian, it follows that each $\hiso(\chi)$, for $\chi \in Z^*$, is a Morita equivalence $\aA$-bimodule \cite[Cor.~3.11]{SchWa15}.
Moreover, these components naturally define Morita self-equivalences over $(\aA,\hG,\alpha \circ q)$, as can readily be checked.
The multiplication in $\hA$ induces, for each $(\chi,\chi') \in Z^* \times Z^*$, a map
\begin{equation}
\label{eq:multiplication}
    m_{(\chi,\chi')} : \hiso(\chi) \otimes_\aA \hiso(\chi') \to \hiso(\chi+\chi'),
    \qquad
    m_{(\chi,\chi')}(s \otimes t) := st,
\end{equation}
which is a $\hG$-equivariant Morita equivalence $\aA$-bimodule isomorphism, \ie, $\hiso(\chi) \otimes_\aA \hiso(\chi')$ and $\hiso(\chi+\chi')$ are equivalent.
These compatibility maps imply that the assignment
\begin{equation}
\label{eq:pic}
    \mathsf{p}_{\hA}: Z^* \to \Pic_{\hG}(\aA),
    \qquad
    \mathsf{p}(\chi) := \left[\hiso(\chi)\right]
\end{equation}
defines a group homomorphism, called the \emph{equivariant Picard homomorphism} associated with~$(\hA,\hG,\ha)$.
As is standard in this setting, it induces a natural $Z^*$-module~structure on $\mathcal{U}(Z(\aA)^G)$ via a group homomorphism 
\begin{equation}
\label{eq:froh}
    \Delta: Z^* \to \Aut(\mathcal{U}(Z(\aA)^G)),
\end{equation}
called the \emph{Fröhlich map} associated with~$(\hA,\hG,\ha)$; see~\cite{Fro73} and~\cite[Cor.~3.94]{BeMa20}.
Composing $\mathsf{p}$ with the forgetful homomorphism $F$ gives the ordinary Picard homomorphism of $(\hA,Z,\theta)$.
Its associated Fröhlich action on $\mathcal{U}(Z(\aA))$ restricts to $\Delta$ on $\mathcal{U}(Z(\aA)^G)$.
We use the same notation for both actions.

At this stage, we fix a \hyperref[sec:facsys]{factor system} $(\hH,\gamma,\omega)$ of $(\hA,Z,\theta)$, by choosing, as described in Section~\ref{sec:free}, for each $\chi \in Z^*$, a finite-dimensional Hilbert space $\hH_\chi$ and an isometry $s(\chi)\in \hA \otimes \End(\C,\hH_\chi)$ such that $\theta_z(s(\chi)) = \chi(z) \cdot s(\chi)$ for all $z \in Z$.
For $0 \in Z^*$, we set $\hH_0:= \mathbb{C}$ and $s(0) := 1_\aA$, as usual.

Given this choice, the isotypic component $\hiso(\chi)$, for $\chi \in Z^*$, can be written, as explained in Section~\ref{sec:free}, as
\begin{equation}
\label{eq:iso}
    \hiso(\chi) = \{y s(\chi) : y\in \aA \otimes \End(\hH_\chi,\C)\}.
\end{equation}
The action $\ha$ on $\hiso(\chi)$ is determined by the continuous map
\begin{equation}
\label{eq:v}
    v(\chi): \hG \to \aA \otimes \End(\hH_\chi),
    \qquad
    v(\chi)(\hg) := v_{\hg}(\chi) := \ha_{\hg}(s(\chi)) s(\chi)^*,
\end{equation}
which satisfies the equivariance relation 
\begin{equation}
\label{eq:v_eq}
    v_{z\hg}(\chi) = \chi(z) \cdot v_{\hg}(\chi)
\end{equation}
for all $z \in Z$ and $\hg \in \hG$.
Importantly, for each $\hg \in \hG$, the element $v_{\hg}(\chi)$ is a partial~isometry with initial projection $s(\chi) s(\chi)^*$ and final projection $\alpha_{q(\hg)}(s(\chi) s(\chi)^*)$.
With this notation, the action $\ha$ on $\hiso(\chi)$ takes the form
\begin{equation}
\label{eq:res_iso}
	\ha_{\hg}(y s(\chi)) 
    =
    \alpha_{q(\hg)}(y) v_{\hg}(\chi) s(\chi)
\end{equation}
for all $\hg \in \hG$ and $y\in \aA \otimes \End(\hH_\chi,\C)$.
Notably, the~\hyperref[eq:froh]{Fr\"ohlich map} is implemented, for
$c \in \mathcal{U}(Z(\aA))$, by
\begin{equation}
\label{eq:froh_def}
    \Delta_\chi^{-1}(c) = s(\chi)^* c\, s(\chi).
\end{equation}
The right-hand side is independent of the choice of factor system, as all factor systems associated with $(\hA,Z,\theta)$ are conjugate (see~Section~\ref{sec:facsys}).

The group
\begin{equation*}
	\Aut_Z(\hA) := \left\{\hfi \in \Aut(\hA) : (\forall z \in Z) \, \theta_z \circ \hfi = \hfi \circ \theta_z\right\}
\end{equation*}
consists of the $Z$-equivariant \Star automorphisms of $\hA$.
Its subgroup
\begin{equation*}
	\Gau_Z(\hA) := \{\hfi \in \Aut_Z(\hA) : \hfi \rvert_{\aA} = \id_\aA\}
\end{equation*}
is the gauge group of $(\hA,Z,\theta)$ (see~\cite[p.~8]{SchWa21}), while $\Aut(\aA)_{[\hA]}$ denotes the group of \Star automorphisms of $\aA$ that can be lifted to $\Aut_Z(\hA)$.
These groups fit into the short exact sequence
\begin{equation*}
	1 \longrightarrow \Gau_Z(\hA) \longrightarrow \Aut_Z(\hA) \longrightarrow \Aut(\aA)_{[\hA]} \longrightarrow 1,
\end{equation*} 

Note that $\theta(Z) \subseteq \Gau_Z(\hA)$.
Furthermore, the assumption that $(\hA,\hG,\ha)$ is a $\hG$-structure guarantees that $\alpha(G) \subseteq \Aut(\aA)_{[\hA]}$---a condition we will make use of in Section~\ref{sec:constructing_G_structures}.

By~\cite[Thm.~4.1]{SchWa21}, $\Aut(\aA)_{\left[\hA\right]}$ can be described in terms of the factor system $(\hH, \gamma, \omega)$ as
\begin{equation}
\label{eq:aut(A)}
	\Aut(\aA)_{[\hA]} = \left\{\varphi \in \Aut(\aA) : (\hH,\gamma,\omega) \sim (\hH,\gamma^\varphi,\omega^\varphi) \right\},
\end{equation}
where $\sim$ is defined in Definition~\ref{def:conjugacy} and the twisted factor system $(\hH,\gamma^\varphi,\omega^\varphi)$ is given by
\begin{equation*} 
    \gamma^\varphi_\chi := \varphi \circ \gamma_\chi \circ \varphi^{-1}
    \qquad
    \text{and}
    \qquad
    \omega^\varphi(\chi,\chi') := \varphi(\omega(\chi,\chi'))
    \label{eq:beta_twist} 
\end{equation*} 
for all $\chi,\chi' \in Z^*$.

We now give a description of the gauge group $\Gau_Z(\hA)$ in terms of data associated with $Z$ and $\aA$.
To phrase our result, we denote by $Z^1_\Delta(Z^*,\mathcal{U}(Z(\aA)))$ the Abelian group of \emph{crossed homomorphisms}, \ie, maps $c: Z^* \to \mathcal{U}(Z(\aA))$ satisfying 
\begin{equation*}
    c(\chi+\chi') = c(\chi) \Delta_\chi(c(\chi'))
\end{equation*}
for all $\chi,\chi' \in Z^*$.
Note that each such map $c$ satisfies $c(0) = 1_\aA$.
With this we wish to establish the following theorem:

\begin{theorem}
\label{thm:gau}
The gauge group $\Gau_Z(\hA)$ is isomorphic to $Z^1_\Delta(Z^*,\mathcal{U}(Z(\aA)))$.
\end{theorem}

Before turning to the proof, we note that Theorem~\ref{thm:gau} implies that $\Gau_Z(\hA)$ is Abelian---a fact that will play a crucial role below.
The proof is organized into two lemmas:

\begin{lemma}
\label{lem:gau1}
Each $\hfi \in \Gau_Z(\hA)$ gives rise to an element $c_{\hfi} \in Z^1_\Delta(Z^*,\mathcal{U}(Z(\aA)))$.
\end{lemma}
\begin{proof}
Let $\hfi \in \Gau_Z(\hA)$, and for each $\chi \in Z^*$, denote by $\hfi_\chi$ its restriction to $\hiso(\chi)$.
Since $\hfi$ is $Z$-equivariant and satisfies $\hfi \rvert_{\aA} = \id_\aA$, it follows that, for each
$\chi \in Z^*$,  the map $\hfi_\chi$ is a Morita equivalence $\aA$-bimodule automorphism of $\hiso(\chi)$.
By~\cite[Prop.~5.4]{SchWa15}, this gives rise to a map $c_{\hfi}: Z^* \to \mathcal{U}(Z(\aA))$, defined via
\begin{equation*}
  c_{\hfi}(\chi)x = \hfi_\chi(x)  
\end{equation*}
for all $x \in \hiso(\chi)$.

Let $\chi,\chi'\in Z^*$, and choose $x \in \hiso(\chi)$ and $y \in \hiso(\chi')$. 
Then:
\begin{equation*}
    c_{\hfi}(\chi+\chi')xy 
    = 
    \hfi_{\chi+\chi'}(xy) 
    = 
    \hfi_\chi(x) \hfi_{\chi'}(y)
    =
    c_{\hfi}(\chi)x c_{\hfi}(\chi')y
    =
    c_{\hfi}(\chi) \Delta_\chi(c_{\hfi}(\chi'))xy.
\end{equation*}
The span of such products $xy$ is dense in $\hiso(\chi+\chi')$, and so this identity extends to all of $\hiso(\chi+\chi')$.
Therefore $c_{\hfi}(\chi+\chi') = c_{\hfi}(\chi) \Delta_\chi(c_{\hfi}(\chi'))$, which completes the proof.
\end{proof}

\begin{lemma}
\label{lem:gau2}
Each $c \in Z^1_\Delta(Z^*,\mathcal{U}(Z(\aA)))$ gives rise to an element $\hfi_c \in \Gau_Z(\hA)$. 
\end{lemma}
\begin{proof}
Let $c \in Z^1_\Delta(Z^*,\mathcal{U}(Z(\aA)))$. 
For each $\chi \in Z^*$, define $\hfi_\chi$ to be the Morita equivalence $\aA$-bimodule automorphism of $\hiso(\chi)$ given by
\begin{equation*}
    \hfi_\chi(x) := c(\chi)x,
    \qquad
    x \in \hiso(\chi).
\end{equation*}
Taking the direct sum of these maps yields a linear map $\hfi_c$ on the dense \Star subalgebra~$\hA_f$ (see~Equation~\eqref{eq:A_f}).
Note that $\hfi_c \rvert_{\aA} = \id_\aA$, because $c(0) = 1_\aA$.
Moreover, it is readily seen that $\hfi_c$ is bijective and satisfies $\theta_z \circ \hfi_c = \hfi_c \circ \theta_z$ for all $z \in Z$.

To prove that $\hfi_c$ is multiplicative, let $\chi, \chi' \in Z^*$, let $x \in \hiso(\chi)$, and let $y \in \hiso(\chi')$.
Then:
\begin{gather*}
    \hfi_c(xy)
    =
    \hfi_{\chi+\chi'}(xy)
    =
    c(\chi+\chi')xy
    =
    c(\chi) \Delta_\chi(c(\chi')) xy
    \\
    =
    c(\chi)x c(\chi')y
    =
    \hfi_\chi(x) \hfi_{\chi'}(y)
    =
    \hfi_c(x) \hfi_c(y),
\end{gather*}
which yields the desired conclusion.

It remains to be shown that $\hfi_c$ is involutive. 
Let $\chi \in Z^*$, and let $x \in \hiso(\chi)$.
The identity $\Delta_\chi(c(-\chi)^*) = c(\chi)$ implies that
\begin{gather*}
    \hfi_c(x^*)
    =
    \hfi_{-\chi}(x^*)
    =
    c(-\chi)x^*
    =
    (xc(-\chi)^*)^*
    \\
    =
    (\Delta_\chi(c(-\chi)^*)x)^*
    =
    (c(\chi)x)^*
    =
    \hfi_\chi(x)^*
    =
    \hfi_c(x)^*,
\end{gather*}
thereby confirming the claim.

In summary, $\hfi_c$ is a \Star algebra automorphism.
The remaining point concerns its continuity.
It is straightforward to verify that $\hfi_c$ is unitary with respect to the inner product defined in Equation~\eqref{eq:innprod}, and thus extends to a unitary operator $U$ on $L^2(\hA)$.
Consequently, there exists a \Star automorphism on $\hA$---again denoted by $\hfi_c$, by a slight abuse of notation---such that $\lambda(\hfi_c(x)) = U \lambda(x) U^*$ for all $x \in \hA$.
This automorphism acts as the identity on $\aA$ and commutes with each $\theta_z$, for $z \in Z$, as is easily checked.
This completes the proof.
\end{proof}

\begin{proof}[Proof of Theorem~\ref{thm:gau}]
A direct computation now shows that the map 
\begin{equation*}
    \Gau_Z(\hA) \to Z^1_\Delta(Z^*,\mathcal{U}(Z(\aA))),
    \qquad
    \hfi \mapsto c_{\hfi}, 
\end{equation*}
is a group isomorphism with inverse 
\begin{equation*}
    Z^1_\Delta(Z^*,\mathcal{U}(Z(\aA))) \to \Gau_Z(\hA),
    \qquad
    c \mapsto \hfi_c.
    \qedhere
\end{equation*}
\end{proof}

\subsection{\texorpdfstring{Classification of $\mathbf{\hG}$-structures}
{Classification of G-hat-structures}}
\label{sec:classification}

Having analyzed $\mathbf{\hG}$-structures in the previous section, we now develop a classification theory for them using group cohomology.

\begin{defn}
Two $\mathbf{\hG}$-structures $(\hA,\hG,\ha)$ and $(\hA',\hG,\ha')$ are said to be~\emph{equivalent} if there exists a $\hG$-equivariant \Star isomorphism $\hfi:\hA \to \hA'$ such that $\hfi \rvert_{\aA} = \id_\aA$.
Such a map is referred to as an \emph{equivalence}.
The equivalence class of a $\hG$-structure $(\hA,\hG,\ha)$ is written $[(\hA,\hG,\ha)]$, and the set of all such equivalence classes is denoted $\Ext(\hG,(\aA,G,\alpha))$.
\end{defn}

Throughout the following, let $\Pic_{\hG}(\aA)$ denote the equivariant Picard group of $(\aA,\hG,\alpha \circ q)$.
It is straightforward to verify that the map
\begin{equation*}
    \Ext(\hG,(\aA,G,\alpha)) \ni [(\hA,\hG,\ha)] \mapsto \mathsf{p}_{\hA} \in \Hom(Z^*,\Pic_{\hG}(\aA)),
\end{equation*}
where $\mathsf{p}_{\hA}$ denotes the equivariant Picard homomorphism associated with $(\hA,\hG,\ha)$ (see Equation~\eqref{eq:pic}), is well-defined, yielding a partition of $\Ext(\hG,(\aA,G,\alpha))$ into the following subsets:

\begin{defn}
For $\mathsf{p} \in \Hom(Z^*,\Pic_{\hG}(\aA))$, we define
\begin{equation*}
    \Ext(\hG,(\aA,G,\alpha),\mathsf{p}) := \{[(\hA,\hG,\ha)] \in \Ext(\hG,(\aA,G,\alpha)) : \mathsf{p}_{\hA} = \mathsf{p}\}.
\end{equation*}  
\end{defn}

For certain group homomorphisms $\mathsf{p}: Z^* \to \Pic_{\hG}(\aA)$, one can construct explicit examples showing that $\Ext(\hG,(\aA,G,\alpha),\mathsf{p})$ is nonempty.
However, in general, this set may be empty.

The construction of $\hG$-structures is addressed in Section~\ref{sec:constructing_G_structures}.
Here, we classify the structures inducing a fixed $\mathsf{p}$, assuming this set is nonempty.

At the outset, recall that each $\mathsf{p} \in \Hom(Z^*,\Pic_{\hG}(\aA))$ induces a $Z^*$-module structure on $\mathcal{U}(Z(\aA)^G)$ via a group homomorphism $\Delta: Z^* \to \Aut(\mathcal{U}(Z(\aA)^G))$ (\cf~Lemma~\ref{lem:MSE:isoME}).
This enables the use of group cohomology, specifically
\begin{equation*}
    Z^2_\Delta(Z^*,\mathcal{U}(Z(\aA)^G)),
    \qquad
    B^2_\Delta(Z^*,\mathcal{U}(Z(\aA)^G)),
    \qquad
    \text{and}
    \qquad
    H^2_\Delta(Z^*,\mathcal{U}(Z(\aA)^G)):
\end{equation*}

\begin{lemma}\label{lem:cocycles}
Suppose $\mathsf{p}:Z^* \to \Pic_{\hG}(\aA)$ is a group homomorphism such that
\begin{equation*}
    \Ext(\hG,(\aA,G,\alpha),\mathsf{p}) \neq \emptyset.
\end{equation*}
For each $\chi \in Z^*$, fix a representative $M(\chi)$ of $\mathsf{p}(\chi)$, with $M(0) := \aA$. 
Then the following~assertions hold:
\begin{enumerate}[label=\arabic*.,ref=\ref{lem:cocycles}.{\arabic*}]
\item
\label{lem:cocycles_rep}
    Each class in $\Ext(\hG,(\aA,G,\alpha),\mathsf{p})$ can be represented by a $\hG$-structure whose $Z$-isotypic components are $M(\chi)$, $\chi \in Z^*$.
\item
\label{lem:cocycles_Z}
    Any two $\hG$-structures whose $Z$-isotypic components are $M(\chi)$, $\chi \in Z^*$, differ by a cocycle $\omega \in Z^2_\Delta(Z^*,\mathcal{U}(Z(\aA)^G))$ in the sense that if $(\hA,\hG,\ha)$ and $(\hA',\hG,\ha')$ are such $\hG$-structures, then their induced multiplication maps satisfy
    \begin{equation}
    \label{eq:multiplication_twist}
        m'_{(\chi,\chi')} = \omega(\chi,\chi') \, m_{(\chi,\chi')} := \phi_{M(\chi + \chi')}(\omega(\chi,\chi')) \, m_{(\chi,\chi')},
    \end{equation}
    for each $(\chi,\chi') \in Z^* \times Z^*$ (see Equation~\eqref{eq:multiplication} and Lemma~\ref{lem:MSE:autME}).
\item
\label{lem:cocycles_B}
    Two $\hG$-structures whose $Z$-isotypic components are $M(\chi)$, $\chi \in Z^*$, represent the same class in $\Ext(\hG,(\aA,G,\alpha),\mathsf{p})$ if and only if the corresponding 2-cocycle $\omega$ from item~2 above is a coboundary, \ie, $\omega \in B^2_\Delta(Z^*,\mathcal{U}(Z(\aA)^G))$.
\end{enumerate}
\end{lemma}
\begin{proof}
\begin{enumerate}[label={\arabic*.}]
\item
    Let $(\hA,\hG,\ha)$ be a $\hG$-structure representing a class in $\Ext(\hG,(\aA,G,\alpha),\mathsf{p})$.
    First, for each $\chi \in Z^*$, fix a $\hG$-equivariant Morita equivalence $\aA$-bimodule isomorphism 
    \begin{equation*}
        \hfi_\chi: \hiso(\chi) \to M(\chi),
    \end{equation*}
    with the family normalized by $\hfi_0:= \id_\aA$.
    Next, restrict $\ha$ to $Z$ to obtain a free C\Star dynamical~system $(\hA,Z,\theta)$ (see Section~\ref{sec:analyzing_G_structures}).  
    By~\cite[Thm.~5.8~(a)]{SchWa15}, there then exists a free C\Star dynamical system $(\hA_M,Z,\theta_M)$ with $\hiso_M(\chi) = M(\chi)$ for each $\chi \in Z^*$, together with a $Z$-equivariant $*$-isomorphism  
    \begin{equation*}
        \hfi: \hA \to \hA_M,
        \qquad 
        \text{such that} 
        \qquad 
        \hfi|_{\hiso(\chi)} = \hfi_\chi.
    \end{equation*}
    Since each $\hfi_\chi$, for $\chi \in Z^*$, is $\hG$-equivariant, $\hfi: \hA \to \hA_M$ intertwines the $\hG$-action $\ha$ on $\hA_f$ (see Equation~\eqref{eq:A_f}) with the canonical $\hG$-action on $\bigoplus_{\chi \in Z^*} M(\chi) \subseteq \hA_M$.
    This motivates the definition
    \begin{equation*}
        \ha_M(\hg)(x) := \hfi(\ha_{\hg}(\hfi^{-1}(x))),
    \end{equation*}
    which yields a strongly continuous action $\ha_M: \hG \to \Aut(\hA_M)$ extending $\theta_M$, with respect to which $\hfi$ is $\hG$-equivariant.
    Thus, the C\Star dynamical system $(\hA_M,\hG,\ha_M)$ defines a $\hG$-structure equivalent to $(\hA,\hG,\ha)$.
\item 
    Let $(\hA,\hG,\ha)$ and $(\hA',\hG,\ha')$ be two $\hG$-structures whose $Z$-isotypic components are $M(\chi)$, $\chi \in Z^*$.
    By Lemma~\ref{lem:MSE:autME}, for each $(\chi,\chi') \in Z^* \times Z^*$, there exists a unique element $\omega(\chi,\chi') \in \mathcal{U}(Z(\aA)^G)$ such that
    \begin{equation*}
        m'_{(\chi,\chi')} = \omega(\chi,\chi') \, m_{(\chi,\chi')}.
    \end{equation*}
    The resulting map $\omega: Z^* \times Z^* \to \mathcal{U}(Z(\aA)^G)$ takes the value~$1_\aA$ whenever one of its arguments is $0$. 
    Moreover, Lemma~\ref{lem:MSE:isoME} implies that, for all $\chi,\chi',\chi'' \in Z^*$,
    \begin{align*}
        &m'_{(\chi,\chi'+\chi'')} \circ \left(\id_\chi \otimes \; m'_{(\chi',\chi'')}\right)
        \\
        &= 
        m'_{(\chi,\chi'+\chi'')} \circ \left(\id_\chi \otimes \; \omega(\chi',\chi'') m_{(\chi',\chi'')}\right) 
        \\
        &= 
        m'_{(\chi,\chi'+\chi'')} \circ \left(\id_\chi \omega(\chi',\chi'') \otimes m_{(\chi',\chi'')}\right) 
        \\
        &= 
        m'_{(\chi,\chi'+\chi'')} \circ \left(\Delta_{M(\chi)}(\omega(\chi',\chi'')) \id_\chi \otimes \; m_{(\chi',\chi'')}\right) 
        \\
        &= 
        \Delta_{M(\chi)}(\omega(\chi',\chi'')) m'_{(\chi,\chi'+\chi'')} \circ \left(\id_\chi \otimes \; m_{(\chi',\chi'')}\right) 
        \\
        &= 
        \Delta_{M(\chi)}(\omega(\chi',\chi'')) \omega(\chi,\chi'+\chi'') m_{(\chi,\chi'+\chi'')} \circ \left(\id_\chi \otimes \; m_{(\chi',\chi'')}\right).
    \end{align*}
    On the other hand, for all $\chi,\chi',\chi'' \in Z^*$,
    \begin{equation*}
        \hspace{3em}
        m'_{(\chi+\chi',\chi'')} \circ \left(m'_{(\chi,\chi')} \otimes \,\id_{\chi''}\right) 
        = 
        \omega(\chi,\chi') \omega(\chi+\chi',\chi'') m_{(\chi+\chi',\chi'')} \circ \left(m_{(\chi,\chi')} \otimes \, \id_{\chi''}\right).
    \end{equation*}
    Hence, $\omega \in Z^2_\Delta(Z^*,\mathcal{U}(Z(\aA)^G))$ as claimed.
\item
    Let $(\hA,\hG,\ha)$ and $(\hA',\hG,\ha')$ be two $\hG$-structures whose $Z$-isotypic components are $M(\chi)$, $\chi \in Z^*$, and let $\omega$ be the corresponding 2-cocycle from item~2.
    
    First, suppose $\omega \in B^2_\Delta(Z^*,\mathcal{U}(Z(\aA)^G))$, \ie, there exists $\varpi \in C^1(Z^*,\mathcal{U}(Z(\aA)^G))$ such that $\omega = d_\Delta\varpi$.
    Using this cochain, one obtains a family 
    \begin{equation*}
        \phi_{M(\chi)}(\varpi(\chi)): M(\chi) \to M(\chi),
        \qquad
        \chi \in Z^*,
    \end{equation*}
    of automorphisms of Morita self-equivalences over $(\aA,\hG,\alpha \circ q)$ implementing an equivalence between $(\hA,\hG,\ha)$ and $(\hA',\hG,\ha')$, as follows from an argument analogous to the final step in the proof of Lemma~\ref{lem:gau2}.
    
    Conversely, suppose that $(\hA,\hG,\ha)$ and $(\hA',\hG,\ha')$ are two equivalent $\hG$-structures whose $Z$-isotypic components are $M(\chi)$, $\chi \in Z^*$.
    Lemma~\ref{lem:MSE:autME} asserts the existence of $\varpi \in C^1(Z^*,\mathcal{U}(Z(\aA)^G))$ implementing this equivalence.
    Moreover, Equation~\eqref{eq:multiplication_twist} implies that $\omega=d_\Delta\varpi$, \ie, $\omega \in B^2_\Delta(Z^*,\mathcal{U}(Z(\aA)^G))$. 
    \qedhere
\end{enumerate}
\end{proof}

For each $\hG$-structure $(\hA,\hG,\ha)$ with associated equivariant Picard homomorphism $\mathsf{p}_{\hA}$ and each $\omega \in Z^2_\Delta(Z^*,\mathcal{U}(Z(\aA)^G))$, let
\begin{equation*}
	\omega \acts (\hA,\hG,\ha)
\end{equation*}
denote the $\hG$-structure obtained by replacing, for each $(\chi,\chi') \in Z^* \times Z^*$, the multiplication map $m_{(\chi,\chi')}$ by
\begin{equation*}
	\phi_{\hiso(\chi+\chi')}(\omega(\chi,\chi')) \circ m_{(\chi,\chi')}.
\end{equation*}
That this again defines a $\hG$-structure with associated equivariant Picard homomorphism $\mathsf{p}_{\hA}$ follows from the cocycle identity and the $G$-invariance of the values of $\omega$.

\begin{corollary}
\label{cor:class1}
Suppose $\mathsf{p}:Z^* \to \Pic_{\hG}(\aA)$ is a group homomorphism such that
\begin{equation*}
    \Ext(\hG,(\aA,G,\alpha),\mathsf{p}) \neq \emptyset.
\end{equation*}
Then the map
\begin{gather*}
	H^2_\Delta(Z^*,\mathcal{U}(Z(\aA)^G)) \times \Ext(\hG,(\aA,G,\alpha),\mathsf{p}) \to \Ext(\hG,(\aA,G,\alpha),\mathsf{p}),
    \\
    \left([\omega],[(\hA,\hG,\ha)]\right) \mapsto 
    [\omega \acts (\hA,\hG,\ha)]
\end{gather*}
is a well-defined simply transitive action.
\end{corollary}

\begin{remark}
    In the case where $G$ is trivial, we recover the classification results from~\cite[Sec.~5]{SchWa15} concerning free C\Star dynamical systems with a given compact Abelian structure group $Z$ and fixed point algebra~$\aA$.
\end{remark}

\subsection{\texorpdfstring{Constructing $\mathbf{\hG}$-structures}
{Constructing G-hat-structures}}
\label{sec:constructing_G_structures}

In this section, we develop our approach for constructing a $\hG$-structure from the given data.
The construction proceeds in four main steps: 
\begin{enumerate}
\item[1.] 
    Construct a candidate algebra $\hA$ lifting $\aA$. 
\item[2.] 
    Lift $\alpha$ to a $\hG$-action on $\hA$ such that $\hA^Z = \aA$ as $G$-algebras. 
\item[3.] 
    Impose continuity of the lifted action. 
\item[4.] 
    Establish that the resulting C\Star dynamical system is free. 
\end{enumerate}

\subsubsection*{Lifting the algebra}

To carry out the first step, we consider a group homomorphism
\begin{equation*}
    \mathsf{p}_0: Z^* \to \Pic(\aA)
\end{equation*}
and determine whether it can be realized as the \emph{ordinary Picard homomorphism} associated with a free C\Star dynamical system $(\hA,Z,\theta)$ over $\aA$ (see~\cite[Sec.~5.1]{SchWa15}).
If so, we take the underlying C\Star algebra~$\hA$ as a lifting of $\aA$.

When starting from an equivariant Picard homomorphism $\mathsf{p}: Z^* \to \Pic_{\hG}(\aA)$, the ordinary datum used in this step is $\mathsf{p}_0 := F \circ \mathsf{p}$.

By~\cite[Thm.~5.14]{SchWa15}, this is possible if and only if the characteristic class of~$\mathsf{p}_0$, 
\begin{equation*}
    \kappa(\mathsf{p}_0) \in H^3_\Delta(Z^*,\mathcal{U}(Z(\aA))),
\end{equation*}
vanishes.
Here, $\Delta: Z^* \to \Aut(\mathcal{U}(Z(\aA)))$ denotes the action~defining the $Z^*$-module~structure on $\mathcal{U}(Z(\aA))$ induced by $\mathsf{p}_0$ (see~\cite[Rem.~5.7]{SchWa15} and the discussion following Equation~\eqref{eq:froh}).

From now on, we assume that the ``if'' part of this characterization holds, allowing us to work with a fixed lifting $(\hA,Z,\theta)$. 
The following associated data is then fixed as well:
\begin{itemize}
\item 
    For each $\chi \in Z^*$, we denote by $\hiso(\chi)$ the corresponding isotypic component.
\item
    We fix a \hyperref[sec:facsys]{factor system} $(\hH,\gamma,\omega)$ of $(\hA,Z,\theta)$, by choosing, for each $\chi \in Z^*$, a finite-dimensional Hilbert space $\hH_\chi$ and an isometry $s(\chi)\in \hA \otimes \End(\C,\hH_\chi)$ such~that $\theta_z(s(\chi)) = \chi(z) \cdot s(\chi)$ for all $z \in Z$.
    For $0 \in Z^*$, we set $\hH_0:= \C$ and $s(0) := 1_\aA$, in accordance with convention.
\end{itemize}

\subsubsection*{Lifting the action}

Within the framework outlined in Step~1, the formulation of the second step requires a slight refinement: 
\begin{enumerate} 
\item[2'.] 
    Lift $\alpha$ to a $\hG$-action on $\hA$ whose restriction to $Z$ coincides with the given action~$\theta$. 
\end{enumerate}
To implement this, we require $\alpha(G) \subseteq \Aut(\aA)_{\left[\hA\right]}$, a necessary condition that can be tested against the factor system $(\hH,\gamma,\omega)$ (see Equation~\eqref{eq:aut(A)}).
Under this assumption, each $\hg \in \hG$ gives rise to an element $\ha_{\hg} \in \Aut_Z(\hA)$ lifting $\alpha_{q(\hg)}$, \ie, $\ha_{\hg} \rvert_{\aA} = \alpha_{q(\hg)}$.

An explicit construction of such a lift is as follows:
Let $g \in G$ and fix $\hg \in q^{-1}(g)$.
Since $\alpha_g \in \Aut(\aA)_{\left[\hA\right]}$, the automorphism $\alpha_g$ admits a lift to an element of $\Aut_Z(\hA)$. 
By~\cite[Thm.~4.1]{SchWa21}, there exists a family of partial isometries $v_{\hg}(\chi) \in \aA \otimes \End(\hH_\chi)$, $\chi \in Z^*$, normalized by $v_{\hg}(0) := 1_\aA$, such that 
\begin{equation*}
    (\hH,\gamma^{\alpha_g},\omega^{\alpha_g}) 
    = 
    v_{\hg}(\hH,\gamma,\omega)v_{\hg}^*.
\end{equation*}
A corresponding lift $\ha_{\hg} \in \Aut_Z(\hA)$ is then determined componentwise as follows: for each $\chi\in Z^*$, its restriction to $\hiso(\chi)$ is given by
\begin{equation}
\label{eq:lift_pre}
	\ha_{\hg}(y s(\chi)) := \alpha_g(y) v_{\hg}(\chi) s(\chi)
\end{equation}
for all $y\in \aA \otimes \End(\hH_\chi,\C)$ (see Equations~\eqref{eq:iso} and~\eqref{eq:res_iso}).
It is immediate that $\ha_{\hg} \rvert_{\aA} = \alpha_g$.

If $\hg'\in q^{-1}(g)$ is another preimage, say $\hg' = z \hg$ for some $z \in Z$, then for each $\chi \in Z^*$ define $v_{\hg'}(\chi) := \chi(z) \cdot v_{\hg}(\chi)$.
This again yields a valid family of partial isometries in $\aA \otimes \End(\hH_\chi)$ and induces an automorphism $\ha_{\hg'} \in \Aut_Z(\hA)$ lifting $\alpha_g$.

Note that this construction along the fibre $q^{-1}(g)$ is independent of the representative,~\ie, 
\begin{equation*}
    v_{z \hg}(\chi) := \chi(z) \cdot v_{\hg}(\chi),
    \qquad
    \text{and}
    \qquad
    \ha_{z\hg} = \ha_z \circ \ha_{\hg} 
\end{equation*}
for all $z \in Z$ and $\hg \in q^{-1}(g)$.

Finally, to guarantee that $\ha_z = \theta_z$ for all $z \in Z$, define, for each $\chi \in Z^*$,
\begin{equation*}
    v_{e_{\hG}}(\chi) := s(\chi) s(\chi)^*.
\end{equation*}

We summarize the construction in the following lemma:

\begin{lemma}
\label{lem:lift}
Suppose that $\alpha(G) \subseteq \Aut(\aA)_{\left[\hA\right]}$.
Then the map
\begin{equation}
\label{eq:lift}
    \ha: \hG \to \Aut(\hA), 
    \qquad 
    \ha(\hg) := \ha_{\hg},
\end{equation}
with each $\ha_{\hg}$ constructed as above, satisfies the following properties:
\begin{enumerate}[label=\arabic*.,ref=\ref{lem:lift}.{\arabic*}]
\item
\label{lem:lift:cond1}
    $\ha_{e_{\hG}}=\id_{\hA}$ and $\ha(\hG) \subseteq \Aut_Z(\hA)$.
\item 
\label{lem:lift:cond2}
    $\ha$ is a lift of $\alpha$, \ie, $\ha_{\hg} \rvert_{\aA} = \alpha_{q(\hg)}$ for all $\hg \in \hG$. 
\item 
\label{lem:lift:cond3}
    The restriction of $\ha$ to $Z$ coincides with $\theta$, \ie, $\ha_z = \theta_z$ for all $z \in Z$.
\item 
\label{lem:lift:z-inv}
    $\ha_{z\hg} = \ha_z \circ \ha_{\hg}$ for all $z \in Z$ and $\hg \in \hG$.
\end{enumerate}
\end{lemma}

Beyond the conclusions of the lemma, it follows directly from the construction that
\begin{equation}
\label{eq:almost-grp-hom}
    \ha_{\hg} \circ \ha_{\hg'} \circ \ha_{\hg\hg'}^{-1} \in \Gau_Z(\hA)
\end{equation}
for all $\hg,\hg' \in \hG$.
This brings us to the central question of whether the map $\ha$ defines a group homomorphism.
Let $\chi \in Z^*$, let $y \in \aA \otimes \End(\hH_\chi,\C)$, and let $\hg,\hg' \in \hG$.
Repeated application of Equation~\eqref{eq:lift_pre} establishes that $\ha_{\hg} \circ \ha_{\hg'}$ acts on $\hiso(\chi)$ as 
\begin{align*}
    (\ha_{\hg} \circ \ha_{\hg'})(y s(\chi)) 
    &=
    \ha_{\hg}\left(\alpha_{q(\hg')}(y) v_{\hg'}(\chi) s(\chi)\right)
    \\
    &= 
    \alpha_{q(\hg)}\left(\alpha_{q(\hg')}(y) v_{\hg'}(\chi)\right) v_{\hg}(\chi) s(\chi)
    \\
    &=
    \alpha_{q(\hg\hg')}(y) \alpha_{q(\hg)}(v_{\hg'}(\chi)) v_{\hg}(\chi) s(\chi),
\end{align*}
On the other hand, we have
\begin{equation*}
    \ha_{\hg\hg'}(y s(\chi)) = \alpha_{q(\hg\hg')}(y) v_{\hg\hg'}(\chi) s(\chi).
\end{equation*}
Comparing these expressions leads to the following lemma:

\begin{lemma}
\label{lem:grp_hom_facsys}
Under the assumptions of Lemma~\ref{lem:lift}, the following statements are equivalent:
\begin{enumerate}[label=(\alph*)]
\item
    $\ha$ is a group homomorphism.
\item 
    For all $\hg,\hg' \in \hG$ and $\chi \in Z^*$, we have $v_{\hg\hg'}(\chi) = \alpha_{q(\hg)}(v_{\hg'}(\chi)) v_{\hg}(\chi)$.
\end{enumerate}
\end{lemma}

We proceed by considering an arbitrary map $\ha: \hG \to \Aut(\hA)$ satisfying Properties~1.--4. in Lemma~\ref{lem:lift}.
Our goal is to identify a cohomological condition under which $\ha: \hG \to \Aut(\hA)$ can be modified to become a group homomorphism.
To this end, consider the map
\begin{equation}
\label{eq:S}
    S: G \to \Aut(\Gau_Z(\hA)),
    \qquad
    S(g)(\hfi) := \ha_{\hg} \circ \hfi \circ \ha_{\hg}^{-1}.
\end{equation}
Well-definedness follows from~\hyperref[lem:lift]{Property~4.}, together with the fact that~$\Gau_Z(\hA)$ is Abelian.
Moreover, Equation~\eqref{eq:almost-grp-hom}, combined with the commutativity of $\Gau_Z(\hA)$,  implies that $S$ is in fact a group homomorphism.
This allows the cohomology group $H^2_S(G, \Gau_Z(\hA))$ to be used in the subsequent analysis.

A similar argument shows that the map
\begin{equation}
\label{eq:delta}
	\delta_{\ha,G} : G \times G \to \Gau_Z(\hA), 
	\qquad
	\delta_{\ha,G}(g,g') := \ha_{\hg} \circ \ha_{\hg'} \circ \ha_{\hg\hg'}^{-1},
\end{equation}
is well-defined and constitutes an $S$-twisted 2-cocycle, \ie, an element of $Z^2_S(G,\Gau_Z(\hA))$.
Moreover, it is straightforward to verify the following independence property:

\begin{lemma}
The class $[\delta_{\ha,G}] \in H^2_S(G, \Gau_Z(\hA))$ is independent of the particular choice of map $\ha: \hG \to \Aut(\hA)$ satisfying Properties~1.--4. in Lemma~\ref{lem:lift}.
\end{lemma}

Let $\delta_{\ha,\hG}$ denote the image of $\delta_{\ha,G}$ under the inflation map
\begin{equation*}
    \text{inf}: C^2(G,\Gau_Z(\hA)) \to C^2(\hG,\Gau_Z(\hA)),
    \qquad
    \text{inf}(\omega) := \omega \circ (q \times q),
\end{equation*}
which induces a homomorphism, denoted by the same symbol, at the level of cohomology:
\begin{equation*}
    \text{inf}: H^2_S(G,\Gau_Z(\hA)) \to H^2_{S \circ q}(\hG,\Gau_Z(\hA)).
\end{equation*}
It follows directly from \hyperref[lem:lift]{Property~4.} that $\delta_{\ha,\hG}$ vanishes whenever either argument lies in $Z$.
Moreover, $\delta_{\ha,\hG}$ determines the pullback extension 
\begin{equation*}
    (\alpha \circ q)^*(\Aut_Z(\hA)) := \{(\hfi,\hat{g}) \in \Aut_Z(\hA) \times \hG:\hfi \rvert_\aA = \alpha_{q(\hg)}\}
\end{equation*}
of $\hG$ by $\Gau_Z(\hA)$.
This extension splits precisely when the class $[\delta_{\ha,\hG}] \in H^2_{S \circ q}(\hG,\Gau_Z(\hA))$ is trivial.

We henceforth assume that this class is trivial, as this is the necessary and sufficient condition for the existence of a group homomorphism $\hG \to \Aut_Z(\hA)$ lifting~$\alpha$.
Then there exists a map $\hfi: \hG \to \Gau_Z(\hA)$ with $\hfi(e_{\hG}) = \id_{\hA}$ such that $\delta_{\ha,\hG} = d_{S \circ q} \hfi$, \ie,
\begin{equation*}
    \delta_{\ha,\hG}(\hg,\hg') = \hfi(\hg) \circ (S \circ q)(\hg)(\hfi(\hg')) \circ \hfi(\hg\hg')^{-1}
\end{equation*}
for all $\hg,\hg' \in \hG$.
With this, we define a family $(\tilde{\alpha}_{\hg})_{\hg \in \hG} \subseteq \Aut_Z(\hA)$ by 
\begin{equation*}
    \tilde{\alpha}_{\hg} := \hfi(\hg)^{-1} \circ \ha_{\hg}.
\end{equation*}
A routine computation shows that
\begin{align*}
    \tilde{\alpha}_{\hg} \circ \tilde{\alpha}_{\hg'} \circ \tilde{\alpha}_{\hg\hg'}^{-1}
    &=
    \hfi(\hg)^{-1} \circ \ha_{\hg} \circ \hfi(\hg')^{-1} \circ \ha_{\hg'} \circ \ha_{\hg\hg'}^{-1} \circ \hfi(\hg\hg')
    \\
    &=
    \hfi(\hg)^{-1} \circ S(q(\hg)) \left(\hfi(\hg')^{-1}\right) \circ \ha_{\hg} \circ \ha_{\hg'} \circ \ha_{\hg\hg'}^{-1} \circ \hfi(\hg\hg')
    \\
    &=
    \hfi(\hg)^{-1} \circ S(q(\hg)) \left(\hfi(\hg')^{-1}\right) \circ \delta_{\ha,\hG}(\hg,\hg') \circ \hfi(\hg\hg')
    \\
    &=
    \delta_{\ha,\hG}(\hg,\hg') \circ \hfi(\hg\hg') \circ S(q(\hg)) \left(\hfi(\hg')^{-1}\right) \circ \hfi(\hg)^{-1}
    =
    \id_{\hA},
\end{align*}
where we have used that $\Gau_Z(\hA)$ is Abelian in the second-to-last equality.
Hence, the map 
\begin{equation*}
    \tilde{\alpha} : \hG \to \Aut(\hA),
    \qquad
    \tilde{\alpha}(\hg) := \tilde{\alpha}_{\hg}
\end{equation*}
is a group homomorphism.
Moreover, Lemma~\ref{lem:lift} remains valid for $\tilde{\alpha}$, except for the third assertion: the restriction of $\tilde{\alpha}$ to $Z$ may fail to coincide with the given action $\theta$.
For all $z \in Z$, we have
\begin{equation*}
    \tilde{\alpha}_z = \hfi(z)^{-1} \circ \theta_z,
\end{equation*}
so that $\tilde{\alpha}_z = \theta_z$ for all $z \in Z$ if and only if $\hfi(z) = \id_{\hA}$ for all $z \in Z$.

Since the obstruction arises from the behavior of $\hfi$ on $Z$, we examine its structure more closely.
For all $\hg \in \hG$ and $z \in Z$, one finds that
\begin{equation*}
    \hfi(\hg z) = \hfi(\hg) \circ (S \circ q)(\hg)(\hfi(z)),
\end{equation*}
a relation that follows from the vanishing of $\delta_{\ha,\hG}$ whenever either argument lies in $Z$, as noted earlier.

Assume that $\hfi(z) = \id_{\hA}$ for all $z \in Z$.
Then the above identity simplifies to $\hfi(\hg z) = \hfi(\hg)$ for all $\hg \in \hG$ and $z \in Z$, so $\hfi$ is constant on $Z$-cosets and thus descends to a well-defined map $\varphi: G \to \Gau_Z(\hA)$ with $\varphi(e_G) = \id_{\hA}$.
In this case, 
\begin{equation*}
    \delta_{\ha,G} = d_S \varphi,
\end{equation*}
\ie, the class $[\delta_{\ha,G}] \in H^2_S(G,\Gau_Z(\hA))$ is trivial.

Conversely, suppose that the class $[\delta_{\ha,G}] \in H^2_S(G,\Gau_Z(\hA))$ is trivial.
Then there exists a map $\varphi: G \to \Gau_Z(\hA)$ with $\varphi(e_G) = \id_{\hA}$ such that $\delta_{\ha,G} = d_S \varphi$.
Applying the inflation map yields
\begin{equation*}
    \delta_{\ha,\hG} = d_{S \circ q} (\varphi \circ q),
\end{equation*}
and the composition $\hfi := \varphi \circ q: \hG \to \Gau_Z(\hA)$ satisfies $\hfi(z) = \id_{\hA}$ for all $z \in Z$.

This correspondence is captured by the following lemma:

\begin{lemma}
\label{lem:grp_hom_class}
For any map $\ha: \hG \to \Aut(\hA)$ satisfying Properties~1.--4. in Lemma~\ref{lem:lift}, the following statements are equivalent:
\begin{enumerate}[label=(\alph*)]
\item
    $\delta_{\ha,\hG} = d_{S \circ q} \hfi$ for some map $\hfi: \hG \to \Gau_Z(\hA)$ satisfying $\hfi(z) = \id_{\hA}$ for all $z \in Z$.
\item
    The class $[\delta_{\ha,G}] \in H^2_S(G,\Gau_Z(\hA))$ is trivial.
\end{enumerate}
\end{lemma}

\begin{corollary}
\label{cor:grp_hom_class}
The following statements are equivalent:
\begin{enumerate}[label=(\alph*)]
\item
    $\alpha$ lifts to a $\hG$-action on $\hA$ whose restriction to $Z$ coincides with the given action~$\theta$. 
\item 
    There exists a map $\ha: \hG \to \Aut(\hA)$ satisfying Properties~1.--4. in Lemma~\ref{lem:lift}, such that the class $[\delta_{\ha,G}] \in H^2_S(G,\Gau_Z(\hA))$ is trivial.
\end{enumerate}
\end{corollary}

The preceding analysis yields the main result of the section, characterizing when the action $\alpha$ lifts compatibly to $\hA$:

\begin{theorem}
\label{thm:grp_hom_class}
Suppose that $\alpha(G) \subseteq \Aut(\aA)_{\left[\hA\right]}$, and let $\ha: \hG \to \Aut(\hA)$ be a map~satisfying Properties~1.--4. in Lemma~\ref{lem:lift}.
If the class 
\begin{equation*}
    [\delta_{\ha,G}] \in H^2_S(G,\Gau_Z(\hA))
\end{equation*}
is trivial, then $\alpha$ lifts to a $\hG$-action on $\hA$ whose restriction to $Z$ coincides with the given action~$\theta$. 
\end{theorem}

\subsubsection{Imposing continuity of the lifted action}

Let $\hat{\alpha}: \hG \to \Aut(\hA)$ be a group homomorphism lifting $\alpha$ and satisfying $\ha \rvert_{Z} = \theta$.
For each $\chi \in Z^*$, this action induces an action on $\hiso(\chi)$ that is implemented by a family of partial~isometries
\begin{equation*}
    (v_{\hg}(\chi))_{\hg \in \hG} \subseteq \aA \otimes \End(\hH_\chi)
\end{equation*}
(see Equations~\eqref{eq:iso},~\eqref{eq:v}, and~\eqref{eq:res_iso}).
In this short section, we ensure the continuity of $\hat{\alpha}$ by proving the following result:

\begin{lemma}
\label{lem:continuity}
Suppose that, for each $\chi \in Z^*$,
\begin{equation*}
    v(\chi): \hG \to \aA \otimes \End(\hH_\chi),
    \qquad
    v(\chi)(\hg) := v_{\hg}(\chi) := \ha_{\hg}(s(\chi)) s(\chi)^*,
\end{equation*}
is continuous.
Then $\hat{\alpha}: \hG \to \Aut(\hA)$ is strongly continuous.
\end{lemma}
\begin{proof}
A standard argument reduces the claim to verifying that, for all $\chi \in Z^*$ and $x \in \hiso(\chi)$, the map $\hG \ni \hg \mapsto \ha_{\hg}(x) \in \hA$
is continuous.

Let $\chi \in Z^*$, let $y \in \aA \otimes \End(\hH_\chi, \C)$, let $\hg, \hg' \in \hG$, and set $g := q(\hg)$ and $g' := q(\hg')$.
Recall that the topology on $\hiso(\chi)$ is induced by the left $\aA$-valued inner product ${}_{\aA}\left\langle x, x' \right\rangle := x x'^*$ (see Section~\ref{sec:cstar}).
Since both $v_{\hg}(\chi)$ and $v_{\hg'}(\chi)$ have initial projection $s(\chi)s(\chi)^*$, we compute:
\begingroup
\allowdisplaybreaks
\begin{align*}
    &{}_{\aA}\left\langle 
    \ha_{\hg}(y s(\chi))-\ha_{\hg'}(y s(\chi)), \ha_{\hg}(y s(\chi))-\ha_{\hg'}(y s(\chi)) \right\rangle
    \\
    &=
    {}_{\aA}\left\langle 
    (\alpha_g(y) v_{\hg}(\chi)-\alpha_{g'}(y) v_{\hg'}(\chi)) s(\chi), 
    (\alpha_g(y) v_{\hg}(\chi)-\alpha_{g'}(y) v_{\hg'}(\chi)) s(\chi) 
    \right\rangle
    \\
    &=
    {}_{\aA}\left\langle 
    \alpha_g(y) v_{\hg}(\chi) - \alpha_{g'}(y) v_{\hg'}(\chi), 
    \alpha_g(y) v_{\hg}(\chi) - \alpha_{g'}(y) v_{\hg'}(\chi) \right\rangle
    \\
    &={}_{\aA}\left\langle 
    \alpha_g(y) (v_{\hg}(\chi) - v_{\hg'}(\chi))
    + 
    (\alpha_g(y) - \alpha_{g'}(y)) v_{\hg'}(\chi),
    \right. 
    \\
    &\hspace{3em}
    \left. 
    \alpha_g(y) (v_{\hg}(\chi) - v_{\hg'}(\chi))
    + 
    (\alpha_g(y) - \alpha_{g'}(y)) v_{\hg'}(\chi) 
    \right\rangle
    \\
    &={}_{\aA}\left\langle
    \alpha_g(y) (v_{\hg}(\chi) - v_{\hg'}(\chi)), 
    \alpha_g(y) (v_{\hg}(\chi) - v_{\hg'}(\chi)) 
    \right\rangle
    \\
    &\hspace{3em}+
    {}_{\aA}\left\langle
    \alpha_g(y) (v_{\hg}(\chi) - v_{\hg'}(\chi)),
    (\alpha_g(y) - \alpha_{g'}(y)) v_{\hg'}(\chi)
    \right\rangle
    \\
    &\hspace{3em}+
    {}_{\aA}\left\langle
    (\alpha_g(y) - \alpha_{g'}(y)) v_{\hg'}(\chi),
    \alpha_g(y) (v_{\hg}(\chi) - v_{\hg'}(\chi))
    \right\rangle
    \\
    &\hspace{3em}+
    {}_{\aA}\left\langle
    (\alpha_g(y) - \alpha_{g'}(y)) v_{\hg'}(\chi),
    (\alpha_g(y) - \alpha_{g'}(y)) v_{\hg'}(\chi)
    \right\rangle
\end{align*}
\endgroup
The claim now follows from the Cauchy--Schwarz inequality, the continuity of $v(\chi)$, the strong continuity of $\alpha$, and the continuity of $q$.
\end{proof}

\subsubsection*{Establishing that the resulting C\Star dynamical system is free}

Assuming that $\alpha$ lifts to a strongly continuous action $\hat{\alpha}: \hG \to \Aut(\hA)$ extending $\theta$ on $Z$, we now establish the following theorem as the final step in the construction:

\begin{theorem}
\label{thm:freeness}
The C\Star dynamical system $(\hA, \hG, \hat{\alpha})$ is free.
\end{theorem}

We begin by noting that the Ellwood map
\begin{equation*}
    \Phi: \hA \otimes_{\text{alg}} \hA \to \Cont(\hG,\hA),
    \qquad
    \Phi(s \otimes t)(\hg) := s \ha_{\hg}(t)
\end{equation*}
is $Z \times Z$-equivariant with respect to the action on $\hA \otimes \hA$ given by
\begin{equation*}
    (z_1,z_2) \acts (s \otimes t) := \theta_{z_1}(s) \otimes \theta_{z_2}(t),
\end{equation*}
and the action on $\Cont(\hG,\hA)$ given by
\begin{equation*}
    ((z_1,z_2) \acts f )(\hg) := \theta_{z_1}\left(f(z_1^{-1}z_2\hg)\right).
\end{equation*}
The isotypic components corresponding to $(\chi, \chi') \in Z^* \times Z^*$ are given by
\begin{equation*}
    (\hA \otimes \hA)(\chi,\chi') = \hiso(\chi) \otimes \hiso(\chi')
    \qquad
    \text{and}
    \qquad
    \Cont(\hG,\hA)(\chi,\chi') = \Cont(\hG)(\chi') \otimes \hiso(\chi+\chi'),
\end{equation*}
where the latter space is defined by
\begin{equation*}
    \Cont(\hG)(\chi') \otimes \hiso(\chi+\chi')
    :=
    \big\{\hf \in \Cont(\hG,\hiso(\chi+\chi')) : (\forall \hg \in \hG)(\forall z \in Z) \, \hf(z\hg)=\chi'(z) \cdot \hf(\hg)\big\}.
\end{equation*}
To establish Theorem~\ref{thm:freeness}, it therefore suffices to prove the following equivariant version:

\begin{theorem}
For each $(\chi,\chi') \in Z^* \times Z^*$, the map
\begin{equation*}
    \Phi_{(\chi,\chi')}: \hiso(\chi) \otimes_{\text{alg}} \hiso(\chi') \to \Cont(\hG)(\chi') \otimes \hiso(\chi+\chi'),
    \qquad
    \Phi_{(\chi,\chi')}(s \otimes t)(\hg) := s \ha_{\hg}(t)
\end{equation*}
has dense range.
\label{thm:freeness_iso}
\end{theorem}

We split the proof of Theorem~\ref{thm:freeness_iso} into a sequence of auxiliary results.
Our first lemma is immediate from the fact that each $s(\chi)$, for $\chi \in Z^*$, is an isometry:

\begin{lemma}
For each $\chi \in Z^*$ the left multiplication
\begin{equation*}
    m_{s(\chi)}: \aA \otimes \End(\hH_\chi,\mathbb{C}) \to \hiso(\chi),
    \qquad
    m_{s(\chi)}(y) := ys(\chi)
\end{equation*}
is surjective.
\label{lem:m_s}
\end{lemma}

To formulate the next result, we recall that for each $\chi \in Z^*$, the action $\ha$ on $\hiso(\chi)$ is implemented by a continuous map 
\begin{equation*}
    v(\chi): \hG \to \aA \otimes \End(\hH_\chi),
    \qquad
    v(\chi)(\hg) := v_{\hg}(\chi),
\end{equation*}
where each $v_{\hg}(\chi)$ is a partial isometry with initial projection $s(\chi) s(\chi)^*$ and final projection $\alpha_{q(\hg)}(s(\chi) s(\chi)^*)$, satisfying the equivariance relation 
\begin{equation*}
   v_{z\hg}(\chi) = \chi(z) \cdot v_{\hg}(\chi) 
\end{equation*}
for all $z \in Z$ and~$\hg \in \hG$ (see Equations~\eqref{eq:iso}--\eqref{eq:res_iso}).

Furthermore, we consider $\Cont(G,\aA \otimes \End(\hH_\chi,\mathbb{C}))$ as a left Hilbert $\Cont(G,\aA)$-module with respect to the canonical left action and the left $\Cont(G,\aA)$-valued inner product defined by
\begin{equation*}
    {}_{\Cont(G,\aA)}\langle f,f' \rangle 
    :=
    f {f'}^*
\end{equation*}
for all $f,f' \in \Cont(G,\aA \otimes \End(\hH_\chi,\mathbb{C}))$.

\begin{lemma}
For each $\chi \in Z^*$ the map
\begin{gather*}
    m_{v(\chi)s(\chi)}: \Cont(G,\aA \otimes \End(\hH_\chi,\mathbb{C})) \to \Cont(\hG)(\chi) \otimes \hiso(\chi),
    \\
    m_{v(\chi)s(\chi)}(f)(\hg) := f(q(\hg)) v_{\hg}(\chi) s(\chi)
\end{gather*}
is a continuous surjection.
\label{lem:mvs}
\end{lemma}
\begin{proof}
Let $\chi \in Z^*$ and let $\hf \in \Cont(\hG)(\chi) \otimes \hiso(\chi)$.
Define
\begin{equation*}
    f:G \to \aA \otimes \End(\hH_\chi,\mathbb{C}),
    \qquad
    f(g) := \hf(\hg) s(\chi)^* v_{\hg}^*(\chi),
\end{equation*}
where $\hg \in q^{-1}(g)$ is any preimage.

We first observe that, although it is not a priori clear that, for each $\hg \in \hG$, the product $\hf(\hg) s(\chi)^* v_{\hg}^*(\chi)$ lies in $\aA \otimes \End(\hH_\chi,\mathbb{C})$, this can be easily deduced from the properties of the individual terms.

To show that $f$ is well-defined, let $\hg' = z\hg$ for some $z \in Z$. 
From the equivariance of $\hf$ and $v(\chi)$, one obtains that
\begin{equation*}
    \hf(\hg') s(\chi)^* v_{\hg'}^*(\chi)
    =
    \hf(z\hg) s(\chi)^* v_{z\hg}^*(\chi)
    =
    \chi(z) \cdot \overline{\chi(z)} \cdot \hf(\hg) s(\chi)^* v_{\hg}^*(\chi)
    =
    f(g),
\end{equation*}
so $f(g)$ is independent of the choice of $\hg \in q^{-1}(g)$.

Since both $\hf$ and $v(\chi)$ are continuous, it follows that $f \in \Cont(G, \aA \otimes \End(\hH_\chi,\mathbb{C}))$.

Surjectivity is established upon verifying that $m_{v(\chi)s(\chi)}(f) = \hf$. 
For any $\hg \in \hG$,
\begin{equation*}
    m_{v(\chi)s(\chi)}(f)(\hg) 
    =
    f(q(\hg)) v_{\hg}(\chi) s(\chi)
    =
    \hf(\hg) s(\chi)^* v_{\hg}^*(\chi) v_{\hg}(\chi) s(\chi).
\end{equation*}
As $v_{\hg}(\chi)$ is a partial isometry with initial projection $s(\chi) s(\chi)^*$, the expression simplifies to $\hf(\hg) s(\chi)^* s(\chi) s(\chi)^* s(\chi) = \hf(\hg)$, which confirms that $m_{v(\chi)s(\chi)}(f) = \hf$, as required.

It remains to prove that $m_{v(\chi)s(\chi)}$ is continuous.
To this end, recall that the topology on $\Cont(\hG)(\chi) \otimes \hiso(\chi)$ is induced by the left $\Cont(G, \aA)$-valued inner product
\begin{equation*}
    {}_{\Cont(G,\aA)}\left\langle \hf,\hf'\right\rangle
    :=
    \hf {\hf'}^*
\end{equation*}
(see Section~\ref{sec:cstar}).
Furthermore, define $p(\chi) \in \Cont(G,\aA \otimes \End(\hH_\chi))$ by
\begin{equation*}
    p(\chi)(g) := \alpha_g(s(\chi) s(\chi)^*).
\end{equation*}

Now, let $f \in \Cont(G,\aA \otimes \End(\hH_\chi,\mathbb{C}))$, let $g \in G$, and let $\hg \in q^{-1}(g)$.
Then
\begin{align*}
    {}_{\Cont(G,\aA)}\langle m_{v(\chi)s(\chi)}(f), m_{v(\chi)s(\chi)}(f) \rangle(g)
    &=
    f(g) v_{\hg}(\chi) s(\chi) s(\chi)^* v_{\hg}(\chi)^* f(g)^*
    \\
    &=
    f(g) v_{\hg}(\chi) v_{\hg}(\chi)^* v_{\hg}(\chi) v_{\hg}(\chi)^* f(g)^*
    \\
    &=
    f(g) v_{\hg}(\chi) v_{\hg}(\chi)^* f(g)^*
    \\
    &=
    f(g) \alpha_g(s(\chi) s(\chi)^*) f(g)^*
    \\
    &=
    {}_{\Cont(G,\aA)}\langle f p(\chi), f p(\chi) \rangle(g),
\end{align*}
where we have used that $v_{\hg}(\chi)$ is a partial isometry with initial projection $s(\chi) s(\chi)^*$ and final projection $\alpha_{q(\hg)}(s(\chi) s(\chi)^*)$.
Hence,
\begin{equation*}
    {}_{\Cont(G,\aA)}\langle m_{v(\chi)s(\chi)}(f), m_{v(\chi)s(\chi)}(f) \rangle = {}_{\Cont(G,\aA)}\langle f p(\chi), f p(\chi) \rangle.
\end{equation*}
By~\cite[Cor.~2.22]{Rae98}, we have the inequality
\begin{equation*}
    {}_{\Cont(G,\aA)}\langle f p(\chi), f p(\chi) \rangle
    \leq
    \sup_{g \in G}\lVert p(\chi)(g) \rVert_{\aA \otimes \End(\hH_\chi)} \cdot {}_{\Cont(G,\aA)}\langle f,f\rangle,
\end{equation*}
which implies that $m_{v(\chi)s(\chi)}$ is continuous.
This completes the proof.
\end{proof}

To proceed, let $\Phi_\alpha: \aA \otimes_{\text{alg}} \aA \to \Cont(G,\aA)$ denote the Ellwood map associated with the free C\Star dynamical system $(\aA,G,\alpha)$.
Also fix $\chi \in Z^*$, and consider the map
\begin{equation*}
    \Phi_\chi: \aA \otimes_{\text{alg}} \hiso(\chi) \to \Cont(\hG)(\chi) \otimes \hiso(\chi),
    \qquad
    \Phi_\chi(a \otimes t)(\hg) := a\ha_{\hg}(t).
\end{equation*}
A direct computation shows that 
\begin{equation*}
     \Phi_\chi \circ \left(\id_{\aA} \otimes\; m_{s(\chi)}\right) 
     = 
     m_{v(\chi)s(\chi)} \circ \left(\Phi_\alpha \otimes \, \id_{\End(\hH_\chi,\mathbb{C})}\right)
\end{equation*}
on $\aA \otimes_{\text{alg}} \aA \otimes \End(\hH_\chi,\mathbb{C})$.
By Lemma~\ref{lem:m_s}, $m_{s(\chi)}$ is surjective, while Lemma~\ref{lem:mvs} ensures that $m_{v(\chi)s(\chi)}$ is continuous and surjective.
Together with the density of the range of $\Phi_\alpha$, this implies that $\Phi_\chi$ has dense range.
We record this for later use:

\begin{corollary}
\label{cor:phichi}
For each $\chi \in Z^*$ the map
\begin{equation*}
    \Phi_\chi: \aA \otimes_{\text{alg}} \hiso(\chi) \to \Cont(\hG)(\chi) \otimes \hiso(\chi),
    \qquad
    \Phi_\chi(a \otimes t)(\hg) := a\ha_{\hg}(t)
\end{equation*}
has dense range.
\end{corollary}

We are now ready to prove Theorem~\ref{thm:freeness_iso}.
The strategy mirrors that of the preceding case. 
To facilitate the argument, we recall that, for each $(\chi,\chi') \in Z^* \times Z^*$, the map
\begin{equation*}
    m_{(\chi,\chi')} : \hiso(\chi) \otimes_{\text{alg}} \hiso(\chi') \to \hiso(\chi+\chi'),
    \qquad
    m_{(\chi,\chi')}(s \otimes t) := st,
\end{equation*}
has dense range (see Equation~\eqref{eq:multiplication}).

\begin{proof}[Proof of Theorem~\ref{thm:freeness_iso}]
Let $(\chi,\chi') \in Z^* \times Z^*$.
A brief computation gives
\begin{equation*}
    \Phi_{(\chi,\chi')} \circ \left(m_{(\chi,0)} \otimes \, \id_{\hiso(\chi')}\right)
    =
    \left(\id_{\Cont(\hG)(\chi')} \otimes \; m_{(\chi,\chi')}\right) \circ \left(\id_{\hiso(\chi)} \otimes \; \Phi_{\chi'} \right)
\end{equation*}
on $\hiso(\chi) \otimes_{\text{alg}} \aA \otimes_{\text{alg}} \hiso(\chi')$.
The map $m_{(\chi,0)}$ is in fact surjective.
For the right-hand side, $\Phi_{\chi'}$ has dense range by Corollary~3.20, while $m_{(\chi,\chi')}$ has dense range and, for each $x \in \hiso(\chi)$, the left multiplication map
\begin{equation*}
	\hiso(\chi') \ni y \mapsto xy \in \hiso(\chi+\chi')
\end{equation*}
is completely bounded.
It follows that $\Phi_{(\chi,\chi')}$ has dense range, as claimed.
\end{proof}



\subsubsection*{The main theorem}

For convenience, we now state the main theorem, which gathers the results established throughout this section.

\begin{theorem}
\label{thm:main}
Let $(\aA,G,\alpha)$ be a free C\Star dynamical system with fixed point algebra $\aB$, and let $1 \rightarrow Z \rightarrow \hG \rightarrow G \rightarrow 1$ be a central extension of compact groups.
Assume the following:
\begin{itemize}
\item
	Let $\mathsf{p}_0: Z^* \to \Pic(\aA)$ be a group homomorphism with trivial characteristic class $\kappa(\mathsf{p}_0)$, and let $(\hA,Z,\theta)$ be a free C\Star dynamical system with fixed point algebra $\aA$ realizing~$\mathsf{p}_0$.
\item 
    Suppose that $\alpha(G) \subseteq \Aut(\aA)_{\left[\hA\right]}$, and let $\ha: \hG \to \Aut(\hA)$ be a map satisfying~Properties~1.--4. in Lemma~\ref{lem:lift}.
    Assume further that the class $$[\delta_{\ha,G}] \in H^2_S(G,\Gau_Z(\hA))$$ is trivial, and denote the induced group homomorphism $\hG \to \Aut(\hA)$ again by $\ha$.
\item 
    Suppose that, for each $\chi \in Z^*$, the map
    \begin{equation*}
    v(\chi): \hG \to \aA \otimes \End(\hH_\chi),
    \qquad
    v(\chi)(\hg) := v_{\hg}(\chi) := \ha_{\hg}(s(\chi)) s(\chi)^*,
\end{equation*}
    is continuous.
    Here $s(\chi)$, $\chi \in Z^*$, is any family of \hyperref[sec:facsys]{freeness-characterizing isometries} for $(\hA,Z,\theta)$.
\end{itemize}
Then $(\hA,\hG,\ha)$ defines a $\hG$-structure.
\end{theorem}

The equivariant Picard homomorphism $\mathsf{p}_{\hA}$ of the resulting $\hG$-structure satisfies $F \circ \mathsf{p}_{\hA} = \mathsf{p}_0$.
Combining Theorem~\ref{thm:main} with Corollary~\ref{cor:class1} yields:

\begin{corollary}
\label{cor:class2}
Under the assumptions of Theorem~\ref{thm:main}, the set $\Ext(\hG,(\aA,G,\alpha),\mathsf{p}_{\hA})$, where $\mathsf{p}_{\hA}$ denotes the equivariant Picard homomorphism associated with $(\hA,\hG,\ha)$ (see Equation~\eqref{eq:pic}), is parametrized by the cohomology group $H^2_\Delta(Z^*,\mathcal{U}(Z(\aA)^G))$.
\end{corollary}

\begin{remark}
While the construction presented above yields a valid $\hG$-structure, another seemingly natural procedure does not.
Given a group homomorphism $\mathsf{p}: Z^* \to \Pic_{\hG}(\aA)$, one may attempt to mimic the procedure in~\cite[Sec.~4]{SchWa15} to construct a C\Star dynamical system $(\hA,\hG,\ha)$. 
However, the resulting C\Star dynamical system generally fails to be free, and thus does not define a genuine $\hG$-structure.
\end{remark}

\section{Examples}
\label{sec:examples}

Having established the general results, we now work out a series of examples that~demonstrate how the abstract framework applies in concrete settings.

\subsection{A simple example}
\label{sec:toyspin}

To introduce the general theory, we begin with a simple example that illustrates the application of the concepts, with an emphasis on classification.
Fix the following data:
\begin{itemize}
\item
    a unital C\Star algebra~$\aB$;
\item
	a central extension $1 \rightarrow Z \rightarrow \hG \xrightarrow{q} G \rightarrow 1$ of compact groups.
\end{itemize}
Set $\aA := \Cont(G,\aB)$ and equip it with the canonical right translation action of $G$,
\begin{equation*}
    (\rho_g \acts f)(g') := f(g'g), 
    \qquad 
    g,g' \in G,
\end{equation*}
yielding a cleft C\Star dynamical system $(\aA,G,\rho)$.
Analogously, consider $\hA := \Cont(\hG,\aB)$~equipped with the canonical right translation action of~$\hG$,
\begin{equation*}
    (\hat{\rho}_{\hg} \acts f)(\hg') := f(\hg'\hg), 
    \qquad 
    \hg,\hg' \in \hG,
\end{equation*}
which gives rise to the cleft C\Star dynamical system $(\hA,\hG,\hat{\rho})$.

$\aA$ embeds into $\hA$ as a C\Star subalgebra via the natural inclusion $f \mapsto f \circ q$.
Under this embedding, ${\hat{\rho}_{\hg}} \rvert_{\aA} = \rho_{q(\hg)}$ for all $\hg \in \hG$, and hence $\rho(G) \subseteq \Aut(\aA)_{[\hA]}$.
Most importantly, $(\hA,\hG,\hat{\rho})$ is precisely a $\hG$-structure for $(\aA,G,\rho)$, as is straightforward to check.

Let $\mathsf{p}_{\hA}$ denote the equivariant Picard homomorphism associated with $(\hA,\hG,\hat{\rho})$ (see Equation~\eqref{eq:pic}).
By Corollary~\ref{cor:class1}, $\Ext(\hG,(\aA,G,\rho),\mathsf{p}_{\hA})$---the set of equivalence classes of $\hG$-structures giving rise to $\mathsf{p}_{\hA}$---corresponds bijectively to $H^2_\Delta(Z^*,\mathcal{U}(Z(\aB)))$.

For example, if $Z = \mathbb{T}$, then $H^2_\Delta(Z^*,\mathcal{U}(Z(\aB)))$ is trivial for any $Z^*$-module structure on $\mathcal{U}(Z(\aB))$ (see, \eg,~\cite[Chap.~VI.6]{MacLane95}). 
Consequently, in this case, there exists---up to equivalence---only one $\hG$-structure with equivariant Picard homomorphism $\mathsf{p}_{\hA}$, namely $(\hA,\hG,\hat{\rho})$.

In contrast, for $Z = \mathbb{T}^2$, $H^2_\Delta(Z^*,\mathcal{U}(Z(\aB)))$ is generally non-trivial (see,~\eg,~\cite[Prop.~6.2]{Wa12}).
For instance, for the trivial $Z^*$-module structure on $\mathcal{U}(Z(\aB))$, it is isomorphic to $\mathcal{U}(Z(\aB))$ (see~\cite[Prop.~6.1]{Wa12}).

\subsection{Coverings of quantum tori}
\label{sec:qtori}

In this section, we show how quantum tori give rise to liftings along finite coverings, thereby providing a natural class of examples for our \hyperref[sec:goal]{lifting problem}.

To set the stage, we briefly recall their basic structure.
Let $n \in \mathbb{N}$ and let $\theta$ be a real skew-symmetric $n \times n$ matrix.  
The \emph{quantum $n$-torus} $\mathbb{T}^n_\theta$ is the universal C\Star algebra generated by unitaries $u_1,\ldots,u_n$ satisfying the relations $u_k u_\ell = \exp(2\pi \imath \theta_{k\ell})\, u_\ell u_k$ for all $1 \leq k, \ell \leq n$.

The matrix $\theta$ is called \emph{quite irrational} if the only $\lambda \in \mathbb{Z}^n$ satisfying
\begin{equation*}
    \exp(2\pi \imath \langle \lambda,\theta\mu \rangle) = 1
\end{equation*}
for all $\mu \in \mathbb{Z}^n$ is $\lambda = 0$.
This is the standard nondegeneracy condition for $\theta$; see~\cite[Def.~1.6]{Phi06}.  
In dimension two, it reduces to the usual irrationality condition on the single deformation parameter, whereas in higher dimensions it expresses an arithmetic nondegeneracy of $\theta$ relative to the lattice $\mathbb{Z}^n$.
In this case, $\mathbb{T}^n_\theta$ is simple and has trivial center, \ie, $Z(\mathbb{T}^n_\theta) = \mathbb{C}$; see~\cite[Thm.~1.9]{Phi06}.

The classical torus $\mathbb{T}^n$ acts on $\mathbb{T}^n_\theta$ via \emph{gauge transformations}, defined by $\gamma_z(u_k):=z_k \cdot u_k$ for all  $z=(z_1,\ldots,z_n) \in \mathbb{T}^n$ and $1 \leq k \leq n$.
The resulting C\Star dynamical system $(\mathbb{T}^n_\theta,\mathbb{T}^n,\gamma)$ is both ergodic and cleft.
For convenience, we identify $\mathbb{T}^n$ with the quotient $\mathbb{R}^n/\mathbb{Z}^n$ and consider the corresponding action of $\mathbb{R}^n/\mathbb{Z}^n$ on $\mathbb{T}^n_\theta$, given by $\gamma_s(u_k) := \exp(2 \pi \imath s_k) \cdot u_k$ for all $s=(s_1,\ldots,s_n) \in \mathbb{R}^n$ and $1 \le k \le n$. 

Now, let $M$ be an invertible $n \times n$ matrix with integer entries, and define $\Gamma := M \cdot \mathbb{Z}^n \subseteq \mathbb{Z}^n$. 
This yields a full-rank lattice $\Gamma$, so that the quotient $\mathbb{R}^n / \Gamma$ is compact.
We thus obtain a central extension of compact Abelian groups:
\begin{equation*}
    1 \longrightarrow \mathbb{Z}^n/\Gamma \longrightarrow \mathbb{R}^n/\Gamma \longrightarrow \mathbb{R}^n/\mathbb{Z}^n \longrightarrow 1.
\end{equation*}
Assume further that $\theta$ is \emph{quite irrational}, and let $\theta'$ be a real skew-symmetric $n \times n$ matrix satisfying
\begin{equation*}
    M \theta' M^T \in \theta + M_n(\mathbb{Z}).
\end{equation*}
As may be concluded from~\cite[Thm.~4.6]{SchWa18}, the C\Star dynamical system $(\mathbb{T}^n_{\theta'},\mathbb{R}^n/\Gamma,\gamma' \circ M^{-1})$ then defines an $\mathbb{R}^n/\Gamma$-structure for $(\mathbb{T}^n_\theta,\mathbb{R}^n/\mathbb{Z}^n,\gamma)$.
More explicitly, let $v_1,\ldots,v_n$ denote the canonical unitary generators of $\mathbb{T}^n_{\theta'}$.
The embedding of $\mathbb{T}^n_\theta$ into $\mathbb{T}^n_{\theta'}$ is determined, for $1 \leq k \leq n$, by
\begin{equation*}
    j(u_k) := v_1^{M_{k1}} \cdots v_n^{M_{kn}}.
\end{equation*}
The condition on $M \theta' M^T$ ensures that these unitaries satisfy the defining relations of $\mathbb{T}^n_\theta$, and simplicity of $\mathbb{T}^n_\theta$ implies that $j$ is injective.
Its image is the fixed point algebra for the restricted $\mathbb{Z}^n/\Gamma$-action.
Moreover, for $s=(s_1,\ldots,s_n)\in\mathbb{R}^n$
and $1 \leq k \le n$,
\begin{equation*}
    \gamma'_{M^{-1}s}(j(u_k))
    =
    e^{2\pi\imath s_k}j(u_k)
    =
    j(\gamma_s(u_k)),
\end{equation*}
so $j$ intertwines $\gamma$ with the induced quotient action on its image.
We identify $\mathbb{T}^n_\theta$ with its image under $j$.

With $\mathsf{p}_{\theta'}$ the associated equivariant Picard homomorphism (see Equation~\eqref{eq:pic}),  Corollary~\ref{cor:class1} provides a classification of $\Ext(\mathbb{R}^n/\Gamma,(\mathbb{T}^n_\theta,\mathbb{R}^n/\mathbb{Z}^n,\gamma),\mathsf{p}_{\theta'})$---the set of equivalence classes of $\mathbb{R}^n/\Gamma$-structures inducing $\mathsf{p}_{\theta'}$---in terms of $H^2_\Delta((\mathbb{Z}^n/\Gamma)^*,\mathbb{T})$.
The coefficient action is trivial. 
Indeed, on each isotypic component, left and right multiplication by a scalar agree. 
The defining relation for the Fröhlich map (see~Equation~\eqref{eq:froh_def}) therefore gives
\begin{equation*}
    \Delta_\chi(\lambda 1_{\mathbb{T}^n_\theta})
    =
    \lambda 1_{\mathbb{T}^n_\theta}
\end{equation*}
for all $\chi \in (\mathbb{Z}^n/\Gamma)^*$ and $\lambda \in \mathbb{T}$.
Thus, the classification uses $H^2((\mathbb{Z}^n/\Gamma)^*,\mathbb{T})$ with trivial coefficients.

For example, let $n=2$, let $M = \operatorname{diag}(2,1)$, and let $\theta' = \theta/2$.
Then $M\theta'M^T = \theta$ and $\mathbb{Z}^2/\Gamma \cong \mathbb{Z}_2$.
The corresponding embedding is given by $j(u_1) = v_1^2$ and $j(u_2)  = v_2$.
Since
\begin{equation*}
    H^2(\mathbb{Z}_2,\mathbb{T})
    \cong\mathbb{T}/\mathbb{T}^2=0,
\end{equation*}
the resulting structure is unique up to equivalence
among those inducing $\mathsf{p}_{\theta'}$.

\subsection{A noncommutative \texorpdfstring{$\Spin(3)$}{Spin(3)}-structure}
\label{sec:connes-landi}

In this section, we develop a noncommutative version of a $\Spin(3)$-structure, demonstrating how the Connes--Landi sphere naturally serves as an example of such a structure for the quantum projective 7-space.
Note that $\Spin(3) \cong \SU(2)$.

We start by reviewing a noncommutative C\Star algebraic version of the classical $\SU(2)$-Hopf fibration over the four-sphere (see~\cite{LaSu05} for a generalization in the context of Hopf--Galois extensions). 
Let $\theta \in \R$ and let $\theta'$ be the skew-symmetric $4 \times 4$-matrix with $\theta_{1,2}' = \theta_{3,4}' = 0$ and $\theta'_{1,3} = \theta'_{1,4} = \theta_{2,3}' = \theta'_{2,4} = \theta/2$. 
The Connes--Landi sphere $\aA(\mathbb S_{\theta'}^7)$ is the universal unital C\Star algebra generated by normal elements $z_1, \dots, z_4$ subject to the relations
\begin{equation*}	
	z_i z_j = e^{2\pi\imath \theta'_{i,j}} \; z_j z_i, 
	\qquad
	z_j^* z_i = e^{2\pi \imath \theta'_{i,j}}\;  z_i z_j^*,
	\qquad
	\sum_{k=1}^4 z_k^* z_k^{} = 1
\end{equation*}
for all $1 \le i,j \le 4$. 

According to~\cite[Expl.~3.5]{SchWa17}, it is equipped with a free action $\alpha$ of $\SU(2)$ given for each $U \in \SU(2)$ on generators by
\begin{equation*}
	\alpha_U: (z_1, \dots, z_4) \mapsto (z_1, \dots, z_4) 
    \begin{pmatrix} U & 0 \\ 0 & U \end{pmatrix}.
\end{equation*}
The corresponding fixed point algebra is the universal unital C\Star algebra $\aA(\mathbb S_\theta^4)$ generated by normal elements $w_1, w_2$ and a central self-adjoint element $x$ satisfying
\begin{equation*}
	w_1 w_2 = e^{2\pi\imath\theta} \; w_2 w_1, 
	\qquad
	w_2^* w_1 = e^{2\pi\imath \theta} \; w_1 w_2^*, 
	\qquad
	w_1^* w_1 + w_2^* w_2 + x^*x = 1.
\end{equation*}

Next, we consider the normal subgroup $N := \{\pm \id_{\mathbb{C}^2}\} \subseteq \SU(2)$. 
By~\cite[Prop.~3.18]{SchWa15}, the induced C\Star dynamical system $(\aA(\mathbb S_{\theta'}^7)^N,\SU(2)/N,\check{\alpha})$, where $\check{\alpha}$ is defined by 
\begin{equation*}
	\check{\alpha}_{[U]} := \alpha_U \rvert_{\aA(\mathbb S_{\theta'}^7)^N}
\end{equation*}
for all $U \in \SU(2)$, remains free.

Set $\aA(\mathbb{P}_{\theta'}^7) := \aA(\mathbb{S}_{\theta'}^7)^N$ and identify $\SO(3)$ with $\SU(2)/N$ via the universal covering map $q: \SU(2) \to \SO(3)$.
Thus, the C\Star dynamical system becomes $(\aA(\mathbb{P}_{\theta'}^7),\SO(3),\check{\alpha} \circ \bar{q}^{-1})$, where $\bar{q}$ is the induced isomorphism from $\SU(2)/N$ to $\SO(3)$.
From this construction, it is evident that $(\aA(\mathbb{S}_{\theta'}^7),\SU(2),\alpha)$ is a $\SU(2)$-structure for $(\aA(\mathbb{P}_{\theta'}^7), \SO(3),\check{\alpha} \circ \bar{q}^{-1})$.

Denote by $\mathsf{p}_{\theta'}$ the associated equivariant Picard homomorphism (see Equation~\eqref{eq:pic}).
Then Corollary~\ref{cor:class1} identifies 
\begin{equation*}
    \Ext(\SU(2),(\aA(\mathbb{P}_{\theta'}^7), \SO(3),\check{\alpha} \circ \bar{q}^{-1}),\mathsf{p}_{\theta'})
\end{equation*}
---the set of equivalence classes of $\SU(2)$-structures yielding $\mathsf{p}_{\theta'}$---with
\begin{equation*}
    H^2_\Delta\left(\mathbb{Z}_2,\mathcal{U}\left(Z(\aA(\mathbb P_{\theta'}^7))^{\SO(3)}\right)\right).
\end{equation*}
In particular, for $\theta = 0$, this simplifies to 
\begin{equation*}
    H^2(\mathbb{Z}_2,\Cont(\mathbb{S}^4,\mathbb{T})) \cong \Cont(\mathbb{S}^4,\mathbb{T}) / \Cont(\mathbb{S}^4,\mathbb{T})^2,
\end{equation*}
which should be compared to Theorem~\ref{thm:topspin}, where only classical spin structures are classified (see also~Section~\ref{sec:topspin}).

\subsection{Noncommutative frame bundles and their spin structures}
\label{sec:quantumspin}

In this section, we present a general framework for noncommutative frame bundles and their spin structures, extending the discussion of Example~\ref{sec:connes-landi}.

Let $(\aA,\SO(n),\alpha)$ be a free C\Star dynamical system.
In analogy with classical frame bundles, such C\Star dynamical systems are referred to, following~\cite{Wa23}, as \emph{noncommutative frame bundles}.

It is worth recalling from~\cite[Cor.~3.13]{Wa23} that noncommutative frame bundles with structure group $\SO(n)$ and fixed point algebra $\aB$ correspond naturally to right Hilbert $\aB$-bimodules that are \emph{tensorial of type~$\pi$}, where $\pi$ is the standard representation of $\SO(n)$ (see~\cite[Def.~3.1 and Def.~3.5]{Wa23}).  
This correspondence assigns to $(\aA,\SO(n),\alpha)$ its associated noncommutative vector bundle with respect to $\pi$, namely
\begin{equation*}
    \Gamma_\aA(\pi) := \{x \in \aA \otimes \mathbb{C}^n : (\forall g \in \SO(n)) \, (\alpha_g \otimes \pi_g)(x) = x\}.
\end{equation*}
In particular, noncommutative frame bundles are uniquely determined, up to isomorphism, by this distinguished class of Hilbert bimodules.  
A concrete example is given by $\Gamma_{\aA(\mathbb P_{\theta'}^7)}(\pi)$; see Example~\ref{sec:connes-landi}.

The basic question is whether $(\aA,\SO(n),\alpha)$ admits a spin structure. 
In concrete cases this can be analyzed using the methods of Section~\ref{sec:constructing_G_structures}. 

Suppose now that $(\hA,\Spin(n),\ha)$ is a spin structure for $(\aA,\SO(n),\alpha)$, and let $\mathsf{p}_{\hA}$ denote the associated equivariant Picard homomorphism (see Equation~\eqref{eq:pic}). 
From Corollary~\ref{cor:class1}, it follows~that
\begin{equation*}
    \Ext(\Spin(n),(\aA,\SO(n),\alpha),\mathsf{p}_{\hA})
\end{equation*}
---the set of equivalence classes of $\Spin(n)$-structures with equivariant Picard homomorphism $\mathsf{p}_{\hA}$---can be described by
\begin{equation*}
    H^2_\Delta(\mathbb{Z}_2,\mathcal{U}(Z(\aA)^{\SO(n)})).
\end{equation*}

\subsection{The classical lifting problem}
\label{sec:topspin}

In this section, we apply our results to classical \hyperref[sec:pb]{principal bundles}, which, however, are not assumed to be smooth or locally trivial, thereby analyzing the~\hyperref[sec:class_sett]{classical lifting problem}.
The content is divided into three parts: we first specify the initial data, then recall the relevant Picard formalism as preparation, and finally discuss the existence–and–classification problem.

Let $P$ be a compact principal $G$-bundle with base space $X$, and let $(\Cont(P),G,\alpha)$ be the induced free C\Star dynamical system (see Section~\ref{sec:free}).
Moreover, let $1\rightarrow Z \rightarrow \hG \rightarrow G \rightarrow 1$ be a central extension of compact groups.

On account of~\cite[Sec.~3]{BrGrRi77}, the Picard group $\Pic(\Cont(P))$ is isomorphic to the semi-direct product $\Pic(P) \rtimes \Homeo(P)$, where $\Pic(P)$ denotes the set of equivalence classes of complex line bundles over $P$ and $\Homeo(P)$ the group of homeomorphisms of $P$.  

For a group homomorphism $\mathsf{p}_0: Z^* \to \Pic(P)$, the $Z^*$-module structure on $\mathcal{U}(\Cont(P)) \cong \Cont(P,\mathbb{T})$, and hence on $\mathcal{U}(\Cont(P)^G) \cong \Cont(X,\mathbb{T})$, is trivial, because the left and right actions by $\Cont(P)$ on sections of complex line bundles coincide. 
This observation will be used tacitly in what follows.

Fix a group homomorphism $\mathsf{p}_0: Z^* \to \Pic(P)$.  
By~\cite[Thm.~6.8]{SchWa15}, a principal $Z$-bundle $\hP$ over $P$ with ordinary Picard homomorphism $\mathsf{p}_0$ exists if and only if the following classes vanish in the indicated order:
\begin{enumerate}
\item[1.]
    The characteristic class $\kappa(\mathsf{p}_0) \in H^3(Z^*,\Cont(P,\mathbb{T}))$,
\item[2.] 
    The secondary characteristic class $\kappa_2(\mathsf{p}_0) \in H^2_{\mathrm{ab}}(Z^*,\Cont(P,\mathbb{T}))$.
\end{enumerate}
Here, $H^2_{\mathrm{ab}}(Z^*,\Cont(P,\mathbb{T}))$ denotes the subgroup of $H^2(Z^*,\Cont(P,\mathbb{T}))$ classifying Abelian extensions of $Z^*$ by $\Cont(P,\mathbb{T})$.
The secondary characteristic class is defined once
$\kappa(\mathsf{p}_0)$ vanishes.

Suppose now that there exists a principal $\hG$-bundle $\hP$ over $X$ such that $\hP/Z=P$ as principal $G$-bundles.
Writing $r: \hP \to P$ for the quotient map, the pullback $f \mapsto f\circ r$ identifies $\Cont(P)$ with the unital C\Star subalgebra $\Cont(\hP)^Z$ of $\Cont(\hP)$.
Under this identification, the induced $G$-action agrees with $\alpha$.

Recall that the existence of such a bundle can be studied using the methods of Section~\ref{sec:constructing_G_structures}, or, in the smooth category, using the Čech cohomology obstruction discussed in~\cite[Sec.~6]{LaWa08}.

Let $\mathsf{p}$ denote its associated equivariant Picard homomorphism (see Equation~\eqref{eq:pic}).
Using the chosen classical lift as a reference point, Corollary~\ref{cor:class1} gives a parametrization of
\begin{equation*}
    \Ext(\hG,(\Cont(P),G,\alpha),\mathsf{p})
\end{equation*}
---the set of equivalence classes of $\hG$-structures realizing $\mathsf{p}$---by $H^2(Z^*,\Cont(X,\mathbb{T}))$.
Furthermore, an argument analogous to~\cite[Cor.~6.9]{SchWa15} shows that the subset
\begin{equation*}
    \Ext_{\mathrm{pb}}(\hG,(\Cont(P),G,\alpha),\mathsf{p}) 
    \subseteq 
    \Ext(\hG,(\Cont(P),G,\alpha),\mathsf{p})
\end{equation*}
consisting of equivalence classes of commutative $\hG$-structures---\ie, principal $\hG$-bundles $\hP$ over $X$ such that $\hP/Z = P$---inducing $\mathsf{p}$ is parametrized by $H^2_{\mathrm{ab}}(Z^*,\Cont(X,\mathbb{T}))$, the subgroup of $H^2(Z^*,\Cont(X,\mathbb{T}))$ classifying Abelian extensions of $Z^*$ by $\Cont(X,\mathbb{T})$.

These conclusions are illustrated in three cases below.

\subsubsection*{Spin~structures}

Let $P$ be a compact smooth principal $\SO(n)$-bundle with base space $X$, where $n \geq 3$, and consider the central extension $1 \to \mathbb{Z}_2 \to \Spin(n) \to \SO(n) \to 1$.

In what follows, we distinguish between \hyperref[sec:spin_structure]{classical spin structures}---smooth and therefore locally trivial---and non-commutative spin structures, as introduced above.

First, note that 
\begin{equation*}
    H^2_{\mathrm{ab}}(\mathbb{Z}_2,\Cont(X,\mathbb{T})) 
    =
    H^2(\mathbb{Z}_2,\Cont(X,\mathbb{T}))
\end{equation*}
in this case, implying that there are no genuinely noncommutative spin structures inducing $\mathsf{p}$.
Moreover, as derived above, the equivalence classes of commutative spin structures inducing $\mathsf{p}$ are parametrized, up to equivalence, by
\begin{equation*}
    H^2_{\mathrm{ab}}(\mathbb{Z}_2,\Cont(X,\mathbb{T})) 
    \cong 
    \Cont(X,\mathbb{T}) / \Cont(X,\mathbb{T})^2
\end{equation*}

If $H^2(X,\mathbb{Z})$ has no $2$-torsion, the coefficient exact sequence yields the natural identification
\begin{equation*}
    \Cont(X,\mathbb{T})/\Cont(X,\mathbb{T})^2
    \cong 
    H^1(X,\mathbb{Z}_2).
\end{equation*}
(see~\cite[Sec.~3.E and Thm.~4.57]{Hat02}).
For compact smooth $X$, this condition is equivalent to $H_1(X,\mathbb{Z})$ having no $2$-torsion; it holds, for example, for spheres, tori, and closed orientable surfaces.

Under this condition, twisting by a principal $\mathbb{Z}_2$-bundle preserves $\mathsf{p}$, since its associated complex line bundle over $X$ has trivial square and is therefore itself trivial.
Thus, all classical spin structures on $P$ induce the same equivariant Picard homomorphism.
The classification at fixed $\mathsf{p}$ consequently recovers their classical parametrization by
$H^1(X,\mathbb{Z}_2)$ (see Theorem~\ref{thm:topspin}).

For general $X$, the classical spin structures inducing $\mathsf{p}$ form a torsor under $\ker\beta$, where $\beta: H^1(X,\mathbb{Z}_2) \to H^2(X,\mathbb{Z})$
is the Bockstein homomorphism associated with 
\begin{equation*}
    0 \longrightarrow \mathbb{Z} \xlongrightarrow{\times 2} \mathbb{Z} \longrightarrow  \mathbb{Z}_2 \longrightarrow 0.
\end{equation*}

\subsubsection*{One-point base space, part~I}

Let $X := \{\ast\}$, the one-point space, and consider the 2-fold covering $1 \rightarrow \mathbb{Z}_2 \rightarrow \hT \rightarrow \mathbb{T} \rightarrow 1$.
Then $\hT$, viewed as a principal $\hT$-bundle over $X$, defines a $\hT$-structure for $P := \mathbb{T}$.  

Let $\mathsf{p}$ denote the equivariant Picard homomorphism of this classical $\hT$-structure.
By Corollary~\ref{cor:class1}, the equivalence classes of $\hT$-structures inducing $\mathsf{p}$ are parametrized by
\begin{equation*}
    H^2(\mathbb{Z}_2,\mathbb{T}) \cong  \mathbb{T}/\mathbb{T}^2 = 0.
\end{equation*}
In consequence, each $\hT$-structure inducing $\mathsf{p}$ is equivalent to the classical structure defined by $\hT$.

\subsubsection*{One-point base space, part~II}
\label{sec:ops2}

Under the same assumptions, but with the central extension 
\begin{equation*}
    1 \longrightarrow \mathbb{T}^2 \longrightarrow \mathbb{T}^3 := \mathbb{T}^2 \times \mathbb{T} \longrightarrow \mathbb{T} \longrightarrow 1,
\end{equation*}
the group $\mathbb{T}^3$, viewed as a principal $\mathbb{T}^3$-bundle over $X$, defines a $\mathbb{T}^3$-structure for $P := \mathbb{T}$.  

Let $\mathsf{p}$ denote the equivariant Picard homomorphism of this classical $\mathbb{T}^3$-structure.
Its underlying ordinary Picard homomorphism $F \circ \mathsf{p}$ is trivial.

The equivalence classes of $\mathbb{T}^3$-structures inducing $\mathsf{p}$ are parametrized by
\begin{equation*}
    H^2(\mathbb{Z}^2,\mathbb{T}) \cong \mathbb{T}.
\end{equation*}
The commutative structures within this family correspond to
\begin{equation*}
    H^2_{\mathrm{ab}}(\mathbb{Z}^2,\mathbb{T}) \cong 0.
\end{equation*}
Thus, this family contains a unique commutative class, represented by the classical $\mathbb{T}^3$-structure.
All other classes in this family are genuinely noncommutative.

\subsection{A noncommutative \texorpdfstring{$\mathbb{T}^3$}{T3}-structure}
\label{sec:QT3}

In the previous example, we observed that the C\Star dynamical system
$(\Cont(\mathbb{T}),\mathbb{T},\rho)$, which corresponds to $\mathbb{T}$ viewed as a principal $\mathbb{T}$-bundle over a point, admits many inequivalent noncommutative
$\mathbb{T}^3$-structures.
In this section, we construct such a noncommutative $\mathbb{T}^3$-structure
using a suitable quantum $3$-torus.

Fix $\vartheta \in \mathbb{R}$ and consider the (degenerate) quantum $3$-torus
$\mathbb{T}^3_\theta$ associated with the skew-symmetric matrix
\[
    \theta :=
    \begin{pmatrix*}[r]
        0 & \vartheta & 0 
        \\
        -\vartheta & 0 & 0
        \\
        0 & 0 & 0
    \end{pmatrix*}.
\]
Thus, $\mathbb{T}^3_\theta$ is the universal C\Star algebra generated by unitaries
$u,v,w$ subject to the relations
\begin{equation*}
    uv = e^{2\pi i \vartheta} vu,
    \qquad
    uw = wu,
    \qquad
    vw = wv.
\end{equation*}
Equivalently, $\mathbb{T}^3_\theta \cong \mathbb{T}^2_{\vartheta} \otimes \Cont(\mathbb{T})$, with $w$ generating the $\Cont(\mathbb{T})$ tensor factor.

As explained in Section~\ref{sec:qtori}, $\mathbb{T}^3_\theta$ carries a
natural action of the classical torus $\mathbb{T}^3$, given by
\begin{equation*}
    \gamma_{(z_1,z_2,z_3)}(u) := z_1 \cdot u,
    \qquad
    \gamma_{(z_1,z_2,z_3)}(v) := z_2 \cdot v,
    \qquad
    \gamma_{(z_1,z_2,z_3)}(w) := z_3 \cdot w.
\end{equation*}
The resulting C\Star dynamical system $(\mathbb{T}^3_\theta,\mathbb{T}^3,\gamma)$ is ergodic and cleft.
Identify $\Cont(\mathbb{T})$ with the unital C\Star subalgebra $C^*(w)$ via $f \mapsto f(w)$.
For the subgroup $\mathbb{T}^2 \times \{1\}$,
\begin{equation*}
    (\mathbb{T}^3_\theta)^{\mathbb{T}^2 \times \{1\}} = C^*(w).
\end{equation*}
Under this identification, the induced quotient action agrees with $\rho$, because it sends $w$ to $zw$ for each $z \in \mathbb{T}$.
Thus, $(\mathbb{T}^3_\theta,\mathbb{T}^3,\gamma)$ defines a $\mathbb{T}^3$-structure for
$(\Cont(\mathbb{T}),\mathbb{T},\rho)$.

We now determine the equivariant Picard homomorphism $\mathsf{p}: \mathbb{Z}^2 \to \Pic_{\mathbb{T}^3}(\Cont(\mathbb{T}))$.
For $(k,l) \in \mathbb{Z}^2$, the corresponding isotypic component is
\begin{equation*}
    \mathbb{T}^3_\theta(k,l)
    = 
    \Cont(\mathbb{T})\,u^k v^l.
\end{equation*}
Since $w$ is central, multiplication by $u^k v^l$ identifies this component, as a $\Cont(\mathbb{T})$-bimodule, with the trivial bimodule $\Cont(\mathbb{T})$.
For $z=(z_1,z_2,z_3) \in \mathbb{T}^3$ and $f \in \Cont(\mathbb{T})$, the action under this identification is
\begin{equation*}
    \mu^{k,l}_z(f)
    = 
    z_1^k z_2^l \rho_{z_3}(f).
\end{equation*}
Consequently, $F \circ \mathsf{p}$ is trivial, while $\mathsf{p}$ records the $\mathbb{T}^2$-characters $(k,l)$.
The corresponding isotypic components of the classical $\mathbb{T}^3$-structure in Example~\ref{sec:ops2} have the same equivariant bimodule structure.
Thus both $(\mathbb{T}^3_\theta,\mathbb{T}^3,\gamma)$ and the classical C\Star dynamical system associated with $\mathbb{T}^3$ represent elements of
$\Ext(\mathbb{T}^3,
(\Cont(\mathbb{T}),\mathbb{T},\rho),\mathsf{p})$.

\subsubsection*{A no-go example}

An earlier version of this paper attempted to use the Heisenberg group
C\Star algebra $C^*(H_3)$ as the basic model for the above construction.
Recall that $C^*(H_3)$ is the universal C\Star algebra generated by unitaries
$u,v,w$ satisfying
\begin{equation*}
    uw=wu,
    \qquad 
    vw=wv, 
    \qquad 
    uv=wvu.
\end{equation*}
There is a strongly continuous action of $\mathbb{T}^2$ given by
\begin{equation*}
    \theta_{(z_1,z_2)}(u) = z_1 \cdot u,
    \qquad
    \theta_{(z_1,z_2)}(v) = z_2 \cdot v,
    \qquad
    \theta_{(z_1,z_2)}(w) = w.
\end{equation*}
The fixed point algebra of the resulting C\Star dynamical system $(C^*(H_3),\mathbb{T}^2,\theta)$ is $\Cont(\mathbb{T})$, which is generated by $w$.
For $(k,l)\in\mathbb{Z}^2$, the corresponding isotypic component is
\begin{equation*}
    C^*(H_3)(k,l)=\Cont(\mathbb{T}) u^k v^l.
\end{equation*}
Each isotypic component contains unitary elements, so $(C^*(H_3),\mathbb{T}^2,\theta)$ is cleft and hence free.
Furthermore, as $w$ is central, each isotypic component represents the neutral element of $\Pic(\Cont(\mathbb{T}))$; consequently the associated ordinary Picard homomorphism $\mathsf{p}_0$ is trivial.

However, the required inclusion
\begin{equation*}
\rho(\mathbb{T}) \subseteq \Aut(\Cont(\mathbb{T}))_{C^*(H_3)}
\end{equation*}
fails.

Indeed, any $\theta$-equivariant \Star automorphism $\varphi$ preserves the isotypic components and therefore satisfies $\varphi(u) = f(w) u$ and $\varphi(v) = g(w) v$ for some
$f,g \in \Cont(\mathbb{T},\mathbb{T})$.
Since $f(w)$ and $g(w)$ are central,
\begin{equation*}
    \varphi(w)
    = 
    \varphi(u) \varphi(v) \varphi(u)^*\varphi(v)^*
    = 
    uvu^*v^*
    = 
    w.
\end{equation*}
For $z \in \mathbb{T} \setminus \{1\}$, the automorphism $\rho_z$ sends $w$ to $zw$ and therefore cannot have a $\theta$-equivariant lift.
Since a $\mathbb{T}^3$-action extending $\theta$ would commute with $\theta$, $C^*(H_3)$ cannot yield the required $\mathbb{T}^3$-structure.


\section*{Open problems}

\begin{itemize}
\item 
    Equivariant Picard groups represent a compelling area for further investigation in C\Star dynamical systems. 
While their definition is well-established, their deeper properties and structure remain underexplored (see~\cite{Ko17,Sad21}, the only relevant sources identified to date). 
Understanding the behavior of these groups, particularly in relation to the lifting problems discussed in this paper, could provide valuable insights into the broader theory of C\Star algebras. 
Consequently, this presents a promising direction for future research.
\item 
   The freeness result in Theorem~\ref{thm:freeness} was proved using ideas from factor system theory.
It remains open whether a simpler or alternative approach is possible.
\end{itemize}

\section*{Acknowledgement}

The author gratefully acknowledges Karl-Hermann Neeb, Kay Schwieger, and Johan Öinert for valuable technical discussions and correspondence that have significantly influenced the development of this work.

\bibliographystyle{abbrv}
\bibliography{short,RS}

\end{document}